\documentclass[a4paper,11pt,reqno]{article}

\usepackage[utf8]{inputenc} 
\usepackage[UKenglish]{babel} 

\usepackage{amsmath, amsfonts}
\usepackage{amsthm, amssymb} 

\usepackage[T1]{fontenc}

\usepackage{lmodern}

\usepackage{mathrsfs} 
\usepackage{faktor} 
\usepackage{booktabs} 

\usepackage{graphicx, subcaption} 

\usepackage[urlcolor=blue, linkcolor=blue, citecolor=blue, colorlinks=true]{hyperref}

\usepackage{fancyhdr} 
\usepackage[left=2cm,right=2cm,top=2cm,bottom=2cm]{geometry}

\usepackage{tikz}
\usetikzlibrary{matrix, arrows, cd, patterns}

\usepackage{afterpage} 

\usepackage{setspace} 

\theoremstyle{plain}
\newtheorem{theorem}{Theorem}[section]

\newtheorem{corollary}[theorem]{Corollary}
\newtheorem{proposition}[theorem]{Proposition}
\newtheorem{lemma}[theorem]{Lemma}
\newtheorem*{remark*}{Remark*}

\theoremstyle{definition}
\newtheorem{definition}[theorem]{Definition}
\newtheorem*{notation}{Notation}

\newtheorem{example}[theorem]{Example} 

\newtheorem{remark}[theorem]{Remark}

\newtheoremstyle{claim}
{0pt}
{-2pt}
{\itshape}
{}
{\bfseries\slshape}
{:}
{.5em }
{ }

\theoremstyle{claim}

\newtheoremstyle{note}
{0pt}
{0pt}
{}
{\parindent}
{\bfseries\slshape}
{:}
{.5em }
{ }

\theoremstyle{note}

\newcommand{\NN}{\mathbb{N}}
\newcommand{\ZZ}{\mathbb{Z}}

\newcommand{\RR}{\mathbb{R}}

\newcommand{\KK}{\mathbb{K}} 

\renewcommand{\AA}{\mathbb{A}} 

\newcommand{\p}{\mathfrak{p}} 
\newcommand{\n}{\mathfrak{n}} 

\newcommand{\F}{\mathscr{F}} 
\newcommand{\G}{\mathscr{G}} 
\renewcommand{\O}{\mathcal{O}} 
\newcommand{\C}{\mathcal{C}} 
\renewcommand{\H}{\mathrm{H}} 
\newcommand{\vC}{\check{\mathcal{C}}} 
\newcommand{\vH}{\check{\mathrm{H}}} 

\newcommand{\U}{\mathscr{U}} 

\newcommand{\midd}{\;\middle\vert\;} 

\newcommand{\ver}{v}

\renewcommand{\restriction}{ | }

\DeclareMathOperator{\Spec}{Spec}
\DeclareMathOperator{\Sym}{Sym}

\DeclareMathOperator{\im}{im}
\DeclareMathOperator{\supp}{supp}

\DeclareMathOperator{\rk}{rk}

\DeclareMathOperator{\Pic}{Pic}

\DeclareMathOperator{\lk}{lk}

\DeclareMathOperator{\red}{red}
\DeclareMathOperator{\tf}{tf}
\DeclareMathOperator{\nil}{nil}
\DeclareMathOperator{\N}{\mathcal{N}}
\DeclareMathOperator{\loc}{loc}
\DeclareMathOperator{\nerve}{nerve}
\DeclareMathOperator{\comb}{comb}
\DeclareMathOperator{\Top}{Top}

\makeindex

\definecolor{grey1}{rgb}{0.3, 0.3, 0.3}
\definecolor{grey2}{rgb}{0.5, 0.5, 0.5}
\definecolor{grey3}{rgb}{0.7, 0.7, 0.7}

    \usepackage[bottom, marginal, multiple]{footmisc} 
    \usepackage{multicol}
    \usepackage{mathtools}
    
    \usepackage{stackengine} 

    \newcommand{\cupdot}{\mathop{\ensurestackMath{\stackinset{c}{}{c}{+.25ex}{\cdot}{\cup}}}}

\begin{document}
\title{Local Picard Group of Pointed Monoids and Their Algebras}

\author{Davide Alberelli, Holger Brenner}


\maketitle

\begin{abstract}
The main goal of this paper is to give an explicit formula for
\[\H^j_{\mathrm{Zar}}(\Spec^\bullet\KK[\triangle], \O^*_{\KK[\triangle]})\]
in terms of simplicial and reduced simplicial cohomology, where $\KK[\triangle]$ is the Stanley-Reisner ring of the simplicial complex $\triangle$. In particular we compute the local Picard group of $\KK[\triangle]$. To achieve this we study the corresponding purely combinatorial problem on the punctured spectrum of the pointed monoid defined by $\triangle$. The cohomology of the sheaf of units on $\Spec^\bullet \KK[\triangle ]$ is then the direct sum of this combinatorial cohomology, which has a decomposition along the vertices, and another part depending on the field $\KK$.
\end{abstract}


\section*{Introduction}

An important invariant of a singularity represented by a local commutative ring $(R, {\mathfrak m})$ is the Picard group of the punctured spectrum $\Spec R \setminus \{ {\mathfrak m} \}$, called the \emph{local Picard group} and denoted by $\Pic^{\loc}(R)$. If $R$ is a normal isolated singularity, this group coincides with the divisor class group, which measures the deviation from factoriality. In this paper we want to compute the local Picard group and its higher cohomological variants in a broad combinatorial setting. We work with a \emph{binoid}, a pointed additively written monoid $(M, +, 0, \infty)$, and the corresponding binoid ring $\KK[M]$ (where $0$ becomes $1$ and $\infty$ becomes $0$). This setting allows to describe zero divisors on the combinatorial level and gives a common framework for monoid rings, Stanley-Reisner rings, toric face rings and nonreduced variants. We will asume that $M$ is finitely generated, commutative, torsion free and cancellative. See \cite{Boettger} for basic properties of this framework, for related concepts see \cite{flores2015homological}, \cite{flores2014picard}, \cite{pena2009mapping}, \cite{lorscheid2012geometry}.

A central question in approaching the local Picard group of $\KK[M]$ is whether it can be computed purely combinatorially and to what extent it depends on the base field $\KK$. On the combinatorial side, we have the finite combinatorial spectrum $\Spec M$, and its punctured variant $\Spec^\bullet M = \Spec M \setminus M_+$, where $M_+ = M \setminus M^*$ is the unique maximal ideal of $M$, which gives rise to the combinatorial local Picard group. A ``line bundle'' on $\Spec^\bullet M$ defines a line bundle on $\Spec^\bullet \KK[ M] $ and so we get a group homomorphism
\[\Pic^{\loc}(M) \longrightarrow \Pic^{\loc}(\KK[M]) \, .  \]
It is known that in the (normal) toric setting this is an isomorphism, see \cite{demeyer1992cohomological}, but it is not true for other combinatorial structures represented by a binoid. To a large extent, we will concentrate on the case of simplicial complexes, simplicial binoids and Stanley-Reisner rings.

The main strategy is based on the fact that binoid generators $x_1, \ldots, x_n$ of the maximal ideal $M_+$ yield an open affine covering $\bigcup D(x_i)$ of  $\Spec^\bullet M$ and also of  $\Spec^\bullet \KK[ M]$, and we want to compute the cohomology of the sheaf of units $\O^*$ in both cases with this covering, via \v{C}ech cohomology. In the first case, this covering is trivially acyclic and so it can be used for computation, provided, we have a good understanding of the units in the localizations of our binoid.

In the algebra case, we have to know that the Picard group and the higher cohomology of the affine pieces $D(x_i)$ are trivial. This is not always true, not even in the nonnormal toric case, but there are many positive results. For example M.\ Pavaman Murthy in 1969~\cite{murthy1969vector} showed that for a positively graded normal ring the Picard group is trivial, Carlo Traverso in 1970~\cite{traverso1970seminormality}, Richard G.\ Swan in 1980~\cite{swan1980seminormality} and David F.\ Anderson in 1981~\cite{anderson1981seminormal} covered the seminormal case, showing that $\Pic(A)=\Pic(A[X_1, \dots, X_n])$ in this case. We will prove that for Stanley-Reisner rings and their localizations, the cohomology of the sheaf of units vanishes.

The next task is to determine the units of $D(x_i)$, where the combinatorial units and the base field have to be considered. Already the affine line shows that there is not a direct splitting of the sheaf of units into combinatorial units and field units. However, in many favorable situations there is such a splitting on the combinatorial topology, the topology generated by the 
$D(x_i)$s, and so the two aspects can be studied separately. In such situations, the first part is determined completely by the combinatorial situation, whereas the second part depends on the constant sheaf given by the units of the field. In the nonintegral case, this part contributes to the local Picard group.

\subsection*{Main results}

We give an overview of our main results, in particular for binoids $M_\triangle$ and their algebra $\KK[\triangle]$ that arise from a simplicial complex $\triangle$. In Lemma~\ref{lemma:intersection-open-subsets} we observe that the intersection pattern of the open subsets $D(x_i)$ of $\Spec M$ is given directly by the faces of the simplicial complex, thus leading us to prove, in Theorem~\ref{thm:simplicial-cohomology}, that the cohomology of a constant sheaf can be computed entirely in terms of simplicial cohomology. In Theorem~\ref{theorem:localization-simplicial-binoid-multiple}, we show that the localization of a simplicial binoid at a face is isomorphic to the smash product of the simplicial binoid of the link of that face and a free group on that face. This opens the door to Theorem~\ref{theorem:semifree-decomposition-of-sheaf}, where we show that we can rewrite the sheaf $\O^*_{M_\triangle}$ as a direct sum of smaller sheaves, indexed by the vertices. These sheaves are actually defined as extensions by zeros of the constant sheaf $\ZZ$ on $D(x_i)$, which is homeomorphic to the spectrum of the link of the corresponding vertex. This brings us to Theorem~\ref{theorem:cohomology-simplicial-complex}, which shows that we can compute sheaf cohomology (and thus the local Picard group) by means of reduced simplicial cohomology, via the formula
\[{\H^j\left(\Spec^\bullet M_\triangle, \O^*_{M_\triangle}\right)}\cong\bigoplus_{\ver \in V}\widetilde{\H}^{j-1}\left(\lk_\triangle( \ver), \ZZ\right)\, .\]
We then use these results to understand the sheaf of units on the binoid algebra and their cohomology. In order to do so, we introduce in Definition~\ref{definition:combinatorial-topology} the combinatorial topology  on $\Spec\KK[M]$, which builds a bridge between the topology on $\Spec M$ and the Zariski topology on $\Spec \KK[M]$. In Proposition~\ref{proposition:split-sheaves-comb-top}, we show that if $M$ is a reduced, torsion-free and cancellative binoid, then we can decompose, in this topology, the sheaf of units of the algebra as a direct sum
\[(\O_{\KK[M]}^*)^{\comb} \cong (i_*\O^*_M)^{\comb}   \oplus    (\KK^*)^{\comb}  \,    ,\]
where $\KK^*$ is the constant sheaf and where $i: \Spec M \rightarrow \Spec K[M]$ is a continuous map which exists under these assumptions. In Theorem~\ref{theorem:push-forward-exact} we prove that the pushforward along $i$ (with respect to the Zariski topology) is exact, thanks to the combinatorial topology. This leads us to prove in Proposition~\ref{proposition:cohomology-pushforward-affine} that the Zariski cohomology of any pushforwarded sheaf vanishes on the affine spectrum $\Spec\KK[M]$.

We then specialize to the case of Stanley-Reisner rings, for which we prove a vanishing result for the affine case in Theorem~\ref{Thm:vanishingspecstanleyreisner}, namely that
\[\H^j(\KK[\triangle], \O^*) = 0\, \]
for all $j\geq 1$, and also for their localizations. This allows us in Corollary~\ref{corollary:covering-cohomology-stanley-reisner}, to compute the cohomomology of the sheaf of units on the punctured spectrum of $K[\triangle]$ with \v{C}ech cohomology of the combinatorial cover. We deduce in Theorem~\ref{theorem:SR-splits-comb-top} from Proposition~\ref{proposition:split-sheaves-comb-top} that for a Stanley-Reisner ring the sheaf of units splits in the combinatorial topology in a constant part and in the pushforward of the combinatorial units.
This leads us to prove Theorem~\ref{theorem:cohomology-stanley-reisner}, that states that we can compute the Zariski cohomology of the sheaf of units entirely in terms of simplicial cohomology, both usual and reduced, as
\[
\H^j(\Spec^\bullet(\KK[\triangle]), \O^*_{\KK[\triangle]}) = \bigoplus_{\ver \in V}\widetilde{\H}^{j-1}(\lk_\triangle(\ver), \ZZ) \oplus   \H^j(\triangle, \KK^*) \, ,\]
for $j\geq0$. Finally, we present in Theorem~\ref{theorem:split-non-reduced-monomial} a generalization to algebras given by any monomial ideal, showing that
\[\H^j(X, \O^*_X) =  \bigoplus_{\ver \in V}\widetilde{\H}^{j-1}(\lk_\triangle(\ver), \ZZ) \oplus    \H^j(\triangle, \KK^*) \oplus \vH^j(\{D(X_i)\}, 1+\N) \, ,\]
where $\triangle$ is the simplicial complex describing the reduction and $\N$ is the coherent sheaf of nilpotent elements.

This work is based on the first author's PhD thesis \cite{alberelli}, which was funded by the DFG-Graduierten\-kolleg Combinatorial Structures in Geometry at the University of Osnabr\"uck. The interested reader will find there all the proofs that are not given here, together with more details and examples. We thank Ilia Pirashvili for his interest and many suggestions.

\section{Binoids and sheaves on their spectra}

\subsection{Definitions}

    For an extended treatment of binoids and their properties, we refer the reader to \cite[Chapter 4]{Boettger}. The spectrum of a binoid $M$, consisting of the prime ideals of $M$ and denoted by $\Spec M$, is naturally a poset ordered by inclusion, with a unique closed point $M_+=M \setminus M^* $, where $M^*$ denotes the group of units. A binoid is called \emph{positive}, if $M^*=\{ 0 \}$. We endow the spectrum of a binoid with the usual Zariski topology, where $S\subseteq \Spec M$ is closed (respectively open) if and only if it is superset closed (respectively subset closed) (see \cite[Remark 4.1.6]{Boettger}). Every prime ideal is generated by a subset of the generators of the binoid. The only open subset that contains the maximal ideal $M_+$ is $\Spec M$.
                
                A binoid is called \emph{integral}, if $M \setminus \{\infty\}$ is a monoid. It is called \emph{cancellative}, if $a+x=a +y \neq \infty$ implies $x=y$. It is called \emph{torsion free}, if $nx=ny\neq \infty$ for some $n \in {\mathbb N}_+$ implies $x=y$. An element $f\in M$ is called \emph{nilpotent} if $nf=\infty$ for some $n\geq 1$. The set of all nilpotent elements will be denoted by $\nil(M)$ and it is easy to show that this is an ideal. We say that $M$ is \emph{reduced} if $\nil(M)=\{\infty\}$.
                 
    A nonzero commutative binoid is called \emph{semifree} with \emph{semibasis} $(a_i)_{i\in I}$ if $M $ is generated by $\{a_i\mid i\in I\}$ and every element $f\in M^\bullet=M \setminus \{ \infty\}$ can be written uniquely as $f=\sum_{i\in I} n_ia_i$, with $n_i=0$ for almost all $i\in I$. The set of $a_i$ such that $n_i\neq 0$ is called the \emph{support} of $f$, $\supp(f)=\{a_i\mid n_i\neq 0\}$.
                
         We assume throughout that $M$ is finitely generated. The following Proposition will be important in describing the stalks of sheaves for schemes of binoids.
                \begin{proposition}\label{proposition:minimal-open-set-prime-ideal}
                    Let $M$ be a binoid. For any prime ideal $\p\in\Spec M$, there exists a unique minimal open set that contains it.
                \end{proposition}
                \begin{proof}
                    Let $x_1, \dots, x_k, x_{k+1}, \dots, x_n$ be the generators of $M_+$. Without loss of generality, we can assume that $\p=\langle x_1, \dots, x_k \rangle$ and that $\p$ does not contain any other generator of $M_+$. Then, $M_\p = M_{x_{k+1}+ \dots+ x_n} $ and                                   
                    \[ \Spec M_\p = \Spec M_{x_{k+1}+ \dots+ x_n} = D(x_{k+1}+ \dots+ x_n)=\{\mathfrak{q}\in \Spec M\mid \{x_{k+1}, \dots, x_n\} \nsubseteq \mathfrak{q}\} 
                    \]
                    is the smallest neighborhood of $\p$.
                \end{proof}

\subsection{Binoid schemes}

                    A \emph{presheaf of binoids} on a topological space $X$ is a contravariant functor from the topology of $X$ to the category of binoids
                    \begin{equation*}
                    \F:\Top_{X}\longrightarrow \mathrm{Bin} \, ,  U \longmapsto  \F(U) \, .
                    \end{equation*}
                           
                    A \emph{sheaf of binoids on $X$} is a presheaf of binoids on $X$ that is a sheaf.

                \begin{definition}\label{definition:binoided-space}
                  A \emph{binoided space} is a pair $(X, \O_X)$ where $X$ is a topological space and $\O_X$ is a sheaf of binoids on $X$, called the \emph{structure sheaf} of the space.
                \end{definition}
                
                Like with rings, we have very special binoided spaces, namely the binoid schemes.
                
                \begin{definition}\label{definition:affine-scheme-binoid}
                    Let $M$ be a binoid. The \emph{affine binoid scheme defined by $M$} is the binoided space $(\Spec M, \O_{\Spec M})$, where $\O_{\Spec M}$ is the sheafification of the presheaf defined on the basis $\{D(f)\}$ as
                    \[
                    \O_{\Spec M}(D(f))=\Gamma(D(f), \O_{\Spec M})=M_f,
                    \]
                    called the \emph{structure sheaf} of $\Spec M$. We usually denote it by $\O_M$.
                \end{definition}
                
                \begin{remark}\label{remark:structure-presheaf}
                    Like for rings, we can explicitly describe the presheaf $\O_{\Spec M}$ as
                    \[
                    \begin{tikzcd}[cramped, cramped, row sep = 0ex,
                    /tikz/column 1/.append style={anchor=base east},
                    /tikz/column 2/.append style={anchor=base west}]
                    D(f_1, \dots, f_r)\rar[mapsto]& \displaystyle\Gamma\left(\bigcup D(f_i), \O_M\right)\\
                    &=\left\{(s_1, \dots, s_r)\in M_{f_1}\times\dots\times M_{f_r}\midd s_i=s_j \text{ in } M_{f_i+f_j}\right\},
                    \end{tikzcd}
                    \]
                    so the image is a subbinoid of $M_{f_1}\times\dots\times M_{f_r}$.
                \end{remark}

                \begin{definition}\label{definition:scheme-binoid}
                    The binoided space $(X, \O_X)$ is a \emph{binoid scheme}, or \emph{scheme of binoids}, if there exists an open cover $\{U_i\}_{i\in I}$ of $X$ and a collection of binoids $\{M_i\}_{i\in I}$, such that
                    \[ U_i \cong \Spec M_i\qquad \text{and} \qquad \O_X \restriction_{U_i} \cong \O_{M_i}\, .                    \]
                \end{definition}
                
               For every open subset $U \subseteq \Spec M$, we get the scheme of binoids $(U, \O_M \restriction_U)$. We are mainly interested in the following scheme of binoids.
               
                \begin{definition}\label{definition:punctured-spectrum-scheme}
                Let $M$ be a binoid. Its \emph{punctured spectrum} is the scheme
                    \[ (\Spec^\bullet M, \O_M\restriction_{\Spec^\bullet M}) \, ,  \]
                    where $\Spec^\bullet M = \Spec M \setminus M_+$.
                \end{definition}

If the $\{x_i\}_{i\in I}$ are the generators of the maximal ideal $M_+$, then there is always the standard covering $\Spec^\bullet M=\cup_{i\in I} D(x_i)$ with the affine binoid schemes $D(x_i)$.
                
                \begin{example}
                    Let $M=(x, y, z \mid x+y=2z)$. Here $M_+=\langle x, y, z\rangle$. Since $\langle x, y\rangle$ is not in $\Spec M$, we can cover the punctured spectrum also by $D(x)$ and $D(y)$ alone. So the standard cover is in general not minimal.
                \end{example}
                
               \begin{proposition}\label{proposition:stalk-binoid-scheme}
                Let $(\Spec M, \O_M)$ be the affine scheme of binoids defined by $M$, let $M_+$ be generated by $x_1, \dots, x_k, x_{k+1}, \dots, x_n$ and let $\p=\langle x_1, \dots, x_k\rangle$ be a prime ideal such that $\p$ does not contain any other generator of $M_+$. The stalk of $\O_M$ at the point $\p$ is
               \[ \O_{M, \p}=M_{\p}=\O_M(D(x_{k+1}+ \dots+ x_n)) = M_{ x_{k+1}+ \dots+ x_n  } \, .  \]
                \end{proposition}
                \begin{proof}
                    This is clear because, thanks to Proposition~\ref{proposition:minimal-open-set-prime-ideal}, we know that $D(x_{k+1}+ \dots+ x_n)$ is the unique minimal open subset of $X$ that contains $\p$.
                \end{proof}

                \begin{proposition}
                    An open subset $W\subseteq \Spec M$ is affine if and only if there exists $f\in M$ such that $W=D(f)$.
                \end{proposition}
                \begin{proof}
                    If $W$ is affine, there exists a binoid $N$ such that $W\cong \Spec N$. In particular, there exists a unique closed point in $W$. This point corresponds to a prime ideal $\p$. Then $W= \{ {\mathfrak q} | {\mathfrak q} \subseteq \p \}$ and $W$ is the open affine subset from Proposition~\ref{proposition:minimal-open-set-prime-ideal}.
                \end{proof}
                
                \begin{remark}
                    An affine open subset $D( {\mathfrak a})$ of $\Spec M$ defines an open affine subset $D( {\mathfrak a}\KK[M])$ of $ \Spec \KK[M]$, but the converse is not true. For example, let $M=(x, y\mid x+y=\infty)$. Then $\Spec^\bullet M=\{\langle x\rangle, \langle y\rangle\}$ is not affine as a binoid scheme, because it does not have a unique closed point. But $ \Spec^\bullet M= \Spec M \setminus \langle x, y\rangle$ is defined by $D(X+Y)\subseteq \Spec^\bullet\KK[M]$. The point is that the element $X+Y$ is not combinatorial, since it involves explicitly the operation $+$ of the ring.
                \end{remark}

\subsection{Sheaves}
                
                From now on, we concentrate on schemes of the type $(U, \O_M\restriction _U)$, where $U$ is an open subset of the affine scheme $\Spec M$, where $M$ is a finitely generated binoid. We refer to such schemes as \emph{quasi-affine} schemes.
                
                An $M$-set is a pointed set $(S, p)$, together with an action of $M$, that satisfies the usual properties (see \cite[Section 1.10]{Boettger}). For any $f\in M$, $S_f$ is an $M_f$-set.

                \begin{definition}\label{definition:sheaf-of-M-sets}
                    Let $(X, \O_X)$ be a binoid scheme. A \emph{sheaf of $\O_X$-sets on $X$}, or \emph{$\O_X$-sheaf}, is a sheaf $\F$ on $X$ such that $\F(U)$ is a $\O_X(U)$-set for any open subset $U$ of $X$, and such that the restrictions are compatible with the actions.
                \end{definition}
                
                \begin{definition}
                    Let $S$ be an $M$-set. The \emph{sheafification of $S$} is the $\O_{\Spec M}$-sheaf $\widetilde{S}$, associated to the presheaf defined on the fundamental open subsets by
                    \[ \widetilde{S}(D(f))=\Gamma(D(f), \widetilde{S})=S_f \, .  \]
                \end{definition}
                
                \begin{remark}
                    Like for binoids, that we saw in Remark~\ref{remark:structure-presheaf}, we can explicitly describe the presheaf $\widetilde{S}$ as
                    \[
                    \begin{tikzcd}[cramped, cramped, row sep = 0ex,
                    /tikz/column 1/.append style={anchor=base east},
                    /tikz/column 2/.append style={anchor=base west}]
                    D(f_1, \dots, f_r)\rar[mapsto]& \displaystyle\Gamma\left(\bigcup D(f_i), \widetilde{S}\right)\subseteq S_{f_1}\times\dots\times S_{f_r},
                    \end{tikzcd}
                    \]    
                    where $M_{f_1}\times \dots \times M_{f_r}$ acts on $S_{f_1}\times\dots\times S_{f_r}$ and we have once again the compatibility conditions on the intersections.
                    Similarly, we can look at the stalk $\widetilde{S}_\p$ at a point $\p\in\Spec M$, and it is easy to see that $\widetilde{S}_\p=S_\p $.
                \end{remark}
                
                \begin{remark}
                 The sheafification of the maximal ideal $M_+$ of $M$ and $\O_M$ are isomorphic as sheaves on the punctured spectrum, i.e. 
                 \[ \widetilde{M_+}\restriction_{\Spec^\bullet M}\cong \O_M\restriction_{\Spec^\bullet M} \, . \]
                 It is not true in general that $\widetilde{M_+}\cong \O_M$ on the whole spectrum, since their global sections are different.
                \end{remark}
                
                \begin{definition}\label{definition:invertible-sheaf}
                Let $(X, \O_X)$ be a binoid scheme and $\F$ a sheaf of $\O_X$-sets on $X$. We say that $\F$ is \emph{locally free of rank $n$} if there exists an $n\in \NN$ and a cover $\{U_i\}_{i\in I}$ of $X$, such that for every $i$
                 \[\F \restriction_{U_i}\cong \left(\O_X\restriction_{U_i}\right)^{\cupdot n} \, ,  \]
                 see \cite[Definition 1.9.2]{Boettger} for the definition of the pointed union of $M$-sets.
                 If $n=1$, we say that the sheaf is \emph{invertible}. We denote by $\mathrm{Loc}_n(X)$ the isomorphism classes of locally free $\O_X$-sheaves of rank $n$.
                \end{definition}

                \begin{remark}
               	Like in the context of schemes (see \cite[Excerise III.5.18]{hartshorne1977algebraic}), there is a correspondence between locally free sheaves and geometric vector bundles on binoid schemes. Starting (locally) with an $M$-set $S$, we get a binoid $\Sym(S)= M \cupdot S \cupdot \Sym^2(S) \cupdot \Sym^3(S) \ldots $, where $\Sym^n(S)$ is $S^{\wedge n}$ modulo the action of the symmetric group $\sum_n$. The spectrum of this binoid is then a geometric realization of $S$.                	
                \end{remark}

                \begin{definition}
                The set of isomorphism classes of locally free sheaves of rank $1$ on $X$, equipped with the operation $\wedge_{\O_X}$ is a group, called the \emph{Picard group} of $X$ and it is denoted by $\Pic(X)$.
                \end{definition}

                \begin{example}\label{example:vector-bundles-of-x+y=2z}
                 Let $M=(x, y, z\mid x+y=2z)$. The sheafification of the ideal $\langle x, z\rangle$, $\widetilde{\langle x, z\rangle}$, is a line bundle on $X=\Spec^\bullet M$. Its order in the group $\Pic(X)$ is 2.
                \end{example}

We will later work mainly  with the cohomological description of the Picard group. For this, we also look at sheaves of abelian groups on a binoid scheme.
                
                \begin{example} Important examples of constant sheaves of groups are $\ZZ$, $(\KK, +)$, $(\KK^*, \cdot)$ for a field $\KK$ and the difference group of a binoid $M$, that we denote by $\Gamma^\bullet$.
                \end{example}
                
                \begin{definition}\label{definition:gamma}
                    Let $M$ be an integral binoid. In this case, the subset $M^\bullet$ is a monoid. The \emph{difference binoid of $M$} is $\Gamma=(-M^\bullet+M)$. The \emph{difference group of $M$} is the group $\Gamma^\bullet$.
                \end{definition}
                
               If $M$ is integral then the constant presheaf is already a sheaf, but if $M$ is not integral, then the sheafification is not trivial.

                \begin{remark}
                    If $M$ is an integral and cancellative binoid, then the map to the difference group is injective and the structure sheaf of Remark~\ref{remark:structure-presheaf} can be defined as
                    \[
                    D(f_1, \dots, f_r)\longmapsto \bigcap_{i=1}^r M_{f_i}\subseteq \Gamma\, .
                    \]
                \end{remark}
                
                \begin{example}
                    Let $M=(x, y\mid x+y=\infty)$. Then $\Spec^\bullet M$ can be covered by the two disjoint open subsets $D(x)$ and $D(y)$. Let $G$ be an abelian group. Then $\Gamma(D(x)\cup D(y), G)=\Gamma(D(x), G) \oplus \Gamma(D(y), G) = G\oplus G$ because $ D(x)\cap D(y) = \emptyset $.
                \end{example}
                
                \begin{remark}
                    Let $U$ be an open subset of $\Spec M$ and let $G$ be a constant sheaf, so the sheafification of a constant presheaf. Then $G(U)=G^k$ where $k$ is the number of connected components of $U$. In the previous example, $U=\Spec^\bullet M$ and $k=2$ because $V(\langle x\rangle)\cap U$ and $V(\langle y\rangle)\cap U$ are the two components.
                \end{remark}
                
                The sheaf we are most interested in is the sheaf of units of a binoid scheme, which is not constant.
                
                \begin{definition}
               Let $(X, \O_X)$ be a binoid scheme. Its \emph{sheaf of units} is the sheaf of abelian groups on $X$
                    \begin{equation*}
                      \O^*_X:\Top_{X} \longrightarrow  \mathrm{Ab} \, , U \longmapsto \left(\O_X(U)\right)^* \, .
                      \end{equation*}
                  Here, given a binoid $M$, $M^*$ denotes the group of its units. If $X=\Spec M$, we denote $\O^*_X$ with $\O^*_M$.
                \end{definition}
                
                The following easy observations were also stated in \cite {PirashviliCohomology} for monoid schemes, see in particular \cite[Proposition 2.2.ii, Lemma 2.4, Proposition 3.1]{PirashviliCohomology}.
                                                
                \begin{theorem}\label{theorem:vanishing-combinatorial-cohomology-affine}
                    Let $M$ be a binoid and $\F$ a sheaf of abelian groups on $\Spec M$. Then
                    \[                    \H^i(\Spec M, \F)=0     \]
                    for any $i\geq 1$.
                \end{theorem}
                                                
                \begin{corollary}
                    Let $(X, \O_X)$ be a separated binoid scheme, $\U$ an affine cover of $X$ and $\F$ a sheaf of abelian groups on $X$. Then $\H^\bullet(X, \F)=\vH^\bullet(\U, \F)$.
                \end{corollary}
                
                \begin{proposition}\label{proposition:constant-sheaf-trivial-cohomology}
                    Let $M$ be an integral binoid. For any open subscheme $U$ of $\Spec M$ and any constant sheaf of abelian groups $\G$ on $U$, we have
                    \[
                    \H^i(U, \G)=0
                    \]
                    for all $i\geq 1$.
                \end{proposition}

                \begin{proposition}\label{proposition:cohomology-vector-bundles}
                    There is an isomorphism of groups
                    \[
                    \Pic(X)\cong \H^1(X, \O^*_X) \, .
                    \]
                \end{proposition}
                
                In our computations, we will mainly work with this characterization of the Picard group, but in some examples we will also present invertible sets or line bundles explicitly.
                
                \begin{definition}\label{definition:local-picard-group}
                    Let $M$ be a binoid. Its \emph{local Picard group} is the Picard group of its punctured spectrum, $\Pic^{\loc}(M)=\Pic(\Spec^\bullet M)$.
                \end{definition}
                
                \begin{remark}
                    If $M$ is torsion-free, cancellative and reduced, the sheaf $\O^*_M$ can be embedded in a flasque sheaf
                    \[
                    \begin{tikzcd}[cramped]
                    \O^*\rar[hook]&\displaystyle\bigoplus_{\begin{subarray}{c}
                        \p \text{ minimal}\\
                        \text{prime ideal of $M$}\end{subarray}} M^*_{\p}\, .
                    \end{tikzcd}
                    \]
                \end{remark}
                
                If $M$ is integral, the sheaf above on the right is just the constant sheaf $\Gamma$.

\subsection{\texorpdfstring{\v{C}}{C}ech-Picard complex}

                In this Section, we are going to study the \v{C}ech complex for the sheaf $\O^*_X$ on the covering of $\Spec^\bullet M$, given by $\{D(x_i)\}$.
                
                \begin{definition}\label{definition:cech-picard-complex}
                    Let $(X, \O_X)$ be a binoid scheme. Let $\U=\{U_i\}_{i\in [n]}$ be a finite affine covering of $X$. The \emph{\v{C}ech-Picard complex of $X$} is the \v{C}ech co-chain complex of $\O_X^*$ with respect to $\U$
                    \begin{equation*}
                    \begin{tikzcd}[cramped, row sep = 0pt]
                    \C(\U, \O_X^*): \C^0(\U, \O_X^*)\rar["\partial^0"] &\C^1(\U, \O_X^*)\rar["\partial^1"] &\dots \rar["\partial^{p-1}"] & \C^p(\U, \O_X^*)\rar["\partial^p"]&\dots \, ,
                    \end{tikzcd}
                    \end{equation*}
                    where the groups are
                    \[
                    \C^p(\U, \O^*_X)=\bigoplus_{1\leq i_0< i_1<\dots<i_p\leq n} \O^*_X\left(U_{i_0}\cap U_{i_1}\cap\dots\cap U_{i_p}\right)
                    \]
                    and the coboundary maps are defined as (see \cite[Section III.4]{hartshorne1977algebraic})
                    \[
                    \left(\partial^{p-1}(\sigma)\right)_{i_0, \dots, i_p}=\sum_{k=0}^{p}(-1)^k\sigma_{i_0, \dots, \widehat{i_k}, \dots, i_p}\restriction_{U_{i_0, \dots, i_p}} \, .
                    \]
                \end{definition}

                \begin{remark}
                	Since $\Spec^\bullet M$ can be covered by $\{D(x_i)\}$ and we know that $\O^*_M$ is acyclic on these affine open subsets, we can compute the local Picard group of $M$ as the first cohomology group of the \v{C}ech-Picard complex on $\{D(x_i)\}$,
                	\[                	\Pic^{\loc} M=\vH^1(\{D(x_i)\}, \O^*) \, .       	\]
                \end{remark}
                
                \begin{example}\label{example:pic-x+y=2z}
                    Let $M=(x, y, z\mid x+y=2z)$ as in Example~\ref{example:vector-bundles-of-x+y=2z} above and let $X=\Spec^\bullet M$. We know that there exists at least an invertible sheaf in $\Pic(X)$, and it has order 2.                    
                    In details, $M_x = (x, -x, y, z \mid x+(-x)=0, y=2z+(-x))$ and so $M_x^*=(x, -x\mid x+(-x)=0)\cong \ZZ$, where this integer represents the coefficient of $x$. Similarly for $y$. On the intersection, when we invert both $x$ and $y$, also $z$ gets inverted. We have
                    \begin{align*}
                    M_{x+y}^*& \cong \ZZ^2 \, ,
                    \end{align*}
                    where $x$ and $z$ are the generators on the right and $y=2z-x$. The maps in the \v{C}ech complex come from the localizations $M_x\stackrel{\iota_y}{\longrightarrow} M_{x+y}$ and $M_y\stackrel{\iota_x}{\longrightarrow} M_{x+y}$ when restricted to the units. So the complex looks like
                    \begin{equation*}
                    \begin{tikzcd}[baseline=(current  bounding  box.center), cramped, row sep = 0,
                    /tikz/column 1/.append style={anchor=base east},
                    /tikz/column 2/.append style={anchor=base west}]
                    M_x^* \oplus M_y^*\rar["\partial^0"]& M_{x+y}^*\rar["\partial^1"] &0 \\
                    \hspace{.4em}\rotatebox[origin=c]{-90}{$\cong$}\hspace{1.8em}\rotatebox[origin=c]{-90}{$\cong$} \hspace{.4em}& \hspace{1.2em}\rotatebox[origin=c]{-90}{$\cong$}\\
                    \hspace{.4em}\ZZ\hspace{.2em}\oplus\hspace{.3em}\ZZ\hspace{.3em} \rar & \ZZ\oplus\ZZ\rar &0 \\
                    (\hspace{.2em}\alpha\hspace{.6em}, \hspace{.8em}\beta\hspace{.2em})\rar[mapsto]&(\alpha-\beta,2\beta)  \, .
                    \end{tikzcd}
                    \end{equation*}
                    To compute the first cohomology, we need to compute the quotient $ \ker(\partial^1) / \im(\partial^0)   \cong M_{x+y}^* / \im(\partial^0) $. The image of $\partial^0$ is generated by $(1,0)$ and $(1, 2)$ as a subgroup of $\ZZ^2$, so the quotient is $ {\ZZ}/{\ZZ}    \oplus     {\ZZ}/ {2\ZZ}   \cong     {\ZZ}/ {2\ZZ}$. So $\Pic(X)\cong  {\ZZ}/ {2\ZZ}$, and we already found a representative of the only non-trivial class in this group, in Example~\ref{example:vector-bundles-of-x+y=2z}.
                \end{example}
                
                \begin{example}
                	Let us consider the binoid $M=(x, y, z, w\mid x+y=z+w)$ and compute its local Picard group. Its punctured spectrum can be covered by the four open subsets $\{D(x), D(y), D(z), D(w)\}$ and the \v{C}ech complex of $\O^*_X$ with respect to this covering looks like
                	\begin{equation*}
                	\begin{tikzcd}[baseline=(current  bounding  box.center), cramped, row sep = 0ex,
                	/tikz/column 1/.append style={anchor=base east},
                	/tikz/column 2/.append style={anchor=base west}]
                	\ZZ^4\rar["\partial^0"] & \ZZ^{14}\rar["\partial^1"]&\ZZ^{12}\rar["\partial^2"] & \ZZ^{3}\rar["\partial^2"] & 0.
                	\end{tikzcd}
                	\end{equation*}
                	One shows directly that $H^2=0$. Since $-4+14-12+3=1$ we know that the rank of $\H^1$ will be 1. It is not hard to examine the relations between elements in the kernel of $\partial^1$ and in the image of $\partial^0$ to conclude that this group has to be free, so, in particular,
                	$\Pic^{\loc}(M)=\ZZ$.
                	A generator of this group is represented by the sheafification of the ideal $\langle x, z\rangle$.
                \end{example}

                \begin{lemma}\label{lemma:M*=M*red}
                    Let $M$ be a binoid. Then $M^*\cong M^*_{\red}$.
                \end{lemma}
                \begin{proof}
                    Let $\varphi_{\red}:M\longrightarrow M_{\red}$ be the reduction morphism. We will prove that it is an isomorphism when restricted to the group of units. We have that $\ker(\varphi_{\red})=\nil(M)$, $\varphi_{\red}$ is a bijection outside this ideal and $\nil(M)\cap M^*=\emptyset$. So we only have to prove that $\varphi_{\red}(M^*)\subseteq M^*_{\red}$, but this is true for any binoid homomorphism.
                \end{proof}
                
                This Lemma proves that, unlike for rings, the nilpotent elements do not play any role in the computation of units of a binoid. We will use this fact later in Section~\ref{section:injections}.
                
                \begin{corollary}
                	Let $(X, \O_X)$ be a binoid scheme. Then $\O^*_X\cong \O^*_{X_{\red}}$.
                \end{corollary}
                
                \begin{example}
                    The above is not true for the torsionfreeification $M_{\tf}$. Let $n\in \NN$, $n\geq 2$ and let $M=(x\mid nx=0)$. Then $M^*\cong \ZZ_n$, but $M_{\tf}\cong\{0, \infty\}$, so $M^*_{\tf}\cong 0$.
                \end{example}

                \begin{remark}
                    Since $\O^*_X\cong \O^*_{X_{\red}}$, also their cohomologies will be the same. So we can concentrate on the reduced case to study the line bundles and the higher cohomology of the sheaf of units. This is not true for $\KK[M ]$
                \end{remark}

                \begin{remark}\label{degreezerounit}
                A binoid is called $\ZZ$-graded if there exists a map $M \setminus \{ \infty \} \rightarrow \ZZ$ compatible with the addition as long as $x+y \neq \infty$. Simplicial binoids are graded. If $M$ is graded, then all its localizations are graded and if $M$ is generated in degree $1$, then we get an exact sequence of sheaves on the punctured spectrum
                \[  0 \longrightarrow \O_0^* \longrightarrow   \O^*   \longrightarrow \ZZ \longrightarrow 0 \, , \]
                where on the right we have the constant sheaf $\ZZ$ and on the left we have the units of degree $0$.
               This last sheaf is the sheaf of units on the corresponding \emph{projective binoid scheme}, which is homeomorphic to the punctured spectrum $U$ but has as its structure sheaf only the degree zero part. If $M$ is positive and $U$ connected, then the corresponding cohomology sequence is
               \[  0 \longrightarrow \ZZ \longrightarrow   \Pic^{\operatorname{proj} } M   \longrightarrow \Pic^{\loc} M \longrightarrow H^1(U, \ZZ)  \longrightarrow \ldots  \, . \]
                \end{remark}

\section{Simplicial Binoids}\label{section:simplcial-binoids}
                
                In this section, we concentrate on the case of binoids arising from simplicial complexes, namely simplicial binoids. We look at the sheaf of groups $\O^*_{M_\triangle}$ restricted to the quasi-affine case and we relate properties of the  \v{C}ech-Picard complex introduced in Definition~\ref{definition:cech-picard-complex} to the simplicial complex. With that, we will provide explicit formulas for the computation of $\H^i(\Spec^\bullet M_\triangle, \O^*_{M_\triangle})$.

                \subsection{The Spectrum of a Simplicial Binoid}
                
                Recall that a \emph{simplicial complex} is a subset $\triangle$ of the power set of the finite \emph{vertex set} $V$ that is closed under taking subsets, i.e.\ $G\in\triangle$ and $F\subseteq G$ implies $F\in\triangle$. Its elements are called \emph{faces} and the maximal faces (under inclusion) are called \emph{facets}. The \emph{dimension} of a face is the number of vertices in it minus 1 and the dimension of $\triangle$ is the maximal dimension of its faces. A \emph{simplicial subcomplex} $\triangle'$ of $\triangle$ is a subset of $\triangle$ that is again a simplicial complex. If $W\subseteq V$ is a subset of the vertices, the \emph{restriction of $\triangle$ to $W$} is $\triangle_W=\{F\in\triangle\mid F\subseteq W\}$. When we say that $\triangle$ is a simplicial complex on $V$, we assume, unless otherwise specified, that the singletons are faces, so $\{v\}\in\triangle$, for every $v\in V$.
                
                \begin{definition}
                    Let $\triangle$ be a simplicial complex on $V$. Its \emph{simplicial binoid} is the binoid with presentation
                    \[                    M_\triangle=\left(x_v, \, v \in V \midd x_H= \sum_{v \in H} x_v =\infty \text{ for }  H \notin \triangle\right) \, .                    \]
                \end{definition}
                
                Of course, the definition of a simplicial binoid is made in such a way that its binoid algebra gives the Stanley-Reisner ring $\KK[\triangle]$. There exists an order-reversing correspondence between faces of the simplicial complex and prime ideals of the binoid (see \cite[Corollary 6.5.13]{Boettger}). In particular, the minimal prime ideals correspond to the (complements of) facets.

            We want to study open subsets of $\Spec M$ in the simplicial case. There is the following correspondence between simplicial subcomplexes of $\triangle$ and closed subsets of $\Spec M_\triangle$.
            \begin{equation*}
            \begin{tikzcd}[baseline=(current  bounding  box.center), row sep = -1em,
            /tikz/column 1/.append style={anchor=base east},
            /tikz/column 2/.append style={anchor=base west},
            /tikz/column 3/.append style={anchor=base west},
            cramped]
            \begin{array}{c}
            \text{Closed subsets}\\
            \text{of }\Spec M_\triangle
            \end{array}\rar[leftrightarrow] &\begin{array}{c}
            \text{Radical ideals of }M_\triangle
            \end{array}\rar[leftrightarrow]  &
            \begin{array}{c}
            \text{Subcomplexes of }\triangle
            \end{array}\\
            V\rar[leftrightarrow] & \begin{array}{c}
            I = \p_1\cap\dots\cap\p_r\\
            \p_i \text{ minimal prime in } V
            \end{array}\rar[leftrightarrow]  &\{F_1, \dots, F_r\} \, .
            \end{tikzcd}
            \end{equation*}
            Here, $\{F_1, \dots, F_r\}$ are the facets of the corresponding simplicial subcomplex. In particular, $V(x_i)$ corresponds to the subsimplicial complex $\triangle'=\triangle_{[n] \setminus \{i\}}$, the restriction of $\triangle$ to $[n] \setminus \{i\}$. Since our goal is to discuss sheaves and, in particular, to build the \v{C}ech complex of the sheaf of units on the punctured spectrum, we are interested in studying open subsets of $\Spec M_\triangle$. We recall that a simplicial binoid is semifree and reduced.
            
            \begin{theorem}[{\cite[Theorem 6.5.8]{Boettger}}]Let $\triangle$ be a simplicial complex on $V$. The binoid $M_\triangle$ is finitely generated by $\#V=n$ elements, semifree and reduced. Conversely, every commutative binoid $M$ satisfying these properties is a simplicial binoid.
            More precisely, $M$ is isomorphic to $M_\triangle$, with $\triangle=\{F\subseteq W\mid\sum_{w\in F}w\neq\infty\}$ for a minimal generating set $W$.
            \end{theorem}
                                              
            In the following Lemma, we denote by $\langle x_1, \dots, \widehat{x_i}, \dots, x_n\rangle$ the prime ideal generated by all the variables except $x_i$.
            
            \begin{lemma}\label{simplicialaxisprimeideal}
                Let $\triangle$ be a simplicial complex on $V=[n]$. Then $\langle x_1, \dots, \widehat{x_i}, \dots, x_n\rangle\in\Spec^\bullet M_\triangle$ for every $i\in V$.
            \end{lemma}
            
            \begin{proposition}\label{proposition:minimal-covering}
                $\{D(x_i)\}$ is a covering by affine subsets of $\Spec^\bullet M_\triangle$ that is minimal among all the possible affine coverings.
            \end{proposition}
            \begin{proof}
                We already know that this is a covering. To prove that it is minimal, it is enough to observe three things. First, thanks to Lemma~\ref{simplicialaxisprimeideal}, we need to cover $\langle x_1, \dots, \widehat{x_i}, \dots, x_n\rangle$.
                Second, $D(x_i)$ is the only affine open subset here that covers $\langle x_1, \ldots, \widehat{x_i}, \ldots, x_n\rangle$.
                Third, any other affine open subset that covers $\langle x_1, \ldots, \widehat{x_i}, \ldots, x_n\rangle$, needs to come from an element $f$ that has the same support of $x_i$, i.e.\ $f=mx_i$, for some $m\in\NN$, and $D(f)=D(x_i)$.
            \end{proof}

            \begin{definition}
            	For $H \subseteq V$, we set
            	\[    D(H):=D(x_H) =D( \sum_{v \in H} x_{v} ) = \bigcap_{v \in H } D(x_v) \, . \]
            \end{definition}

            The given simplicial complex has a direct effect on the intersection pattern in $\Spec M_\triangle$.
            
            \begin{lemma}\label{lemma:intersection-open-subsets}
                Let $\{D(x_i)\}$ be the covering of $\Spec^\bullet M_\triangle$ as above. Then $D(H)\neq \emptyset $ if and only if $H$ is a face of $\triangle$.
            \end{lemma}
            \begin{proof}
            	In terms of faces, which are in correspondence to the prime ideals of $M_\triangle$, we have
            	$D(H) = \{F \in \triangle| H \subseteq F \}$, from where the result follows.
            \end{proof}

            \begin{corollary}\label{corollary:cech-covering-nerve}
                The nerve of the covering of $\Spec^\bullet M_\triangle$ given by $\{D(x_i)\}$ is the simplicial complex itself.
            \end{corollary}
                        
            The following Lemma generalizes Proposition~\ref{proposition:minimal-covering}.
            
            \begin{lemma}\label{lemma:minimal-covering}
                Let $U\subseteq\Spec^\bullet(M_\triangle)$ be an open subset. Then $U$ can be minimally covered by $\mathscr{V}=\{D(F_1), \dots, D(F_j)\}$, for some $F_1, \dots, F_j\in\triangle$.
            \end{lemma}
            \begin{proof}
                Let $\p_1, \dots, \p_k$ be the maximal prime ideals in $U$. Then $F_1, \dots, F_j$ are the faces corresponding to these prime ideals. Minimality can be proved as in Proposition~\ref{proposition:minimal-covering}.
            \end{proof}

            We explore now the relation between \v{C}ech cohomology on the covering $\{D(x_i)\}$ of a constant sheaf of abelian groups on $\Spec^\bullet M_\triangle$ and simplicial cohomology of $\triangle$ with coefficients in that group.
            
            \begin{remark}\label{remarkcomponent}
                Let $G$ be a constant sheaf of abelian groups on $\Spec M$. We already know that
                \[    \Gamma(U, G)=  G^{\#\{\text{connected components of }U\}} \, .    \]
                Moreover, if $\{D(x_i)\}$ is the usual covering of the punctured spectrum given by the combinatorial open subsets, we can use it to compute the cohomology via \v{C}ech cohomology, and we can explicitly write the groups in the \v{C}ech complex as
                \[ \vC\left(D(X_{i_0})\cap \dots\cap D(x_{i_k}), G\right)= \left\{\begin{aligned}
                &G,    & \text{ if } D(X_{i_0})\cap \dots\cap D(x_{i_k})\neq \emptyset ,\\
                &0,                        & \text{ otherwise},
                \end{aligned}\right. \,  \]
                because this intersection is either empty or connected.
            \end{remark}
                            
            \begin{theorem}\label{thm:simplicial-cohomology}
                Let $\triangle$ be a simplicial complex on $V=[n]$. Let $\{D(x_i), 1\leq i\leq n\}$ be the usual acyclic covering of $\Spec^\bullet M_\triangle$ and let $G$ be the constant group sheaf on this space. Then \v{C}ech cohomology and simplicial cohomology are described by the same chain complexes
                \[ \C^\bullet(\triangle, G)=\vC^\bullet(\{D(x_i)\}, G )\, . \]               
                In particular, the cohomology groups are the same
                \[   \H^i(\triangle, G)=\vH^i(\{D(x_i)\}, G)      \]
                for all $i\geq 0$.
            \end{theorem}
            \begin{proof}
                We understand what we have on the right. Thanks to Lemma~\ref{lemma:intersection-open-subsets} and Remark~\ref{remarkcomponent}, $\vC^j(\{D(x_i)\}, G)=G^{\triangle_j}$, i.e. we have a $G$ for any face in $\triangle$ of dimension $j$. The maps are the usual maps of the \v{C}ech complex.
                
                On the left hand side, we can follow a reasoning similar to \cite[Section 1.3]{miller2005combinatorial}.                
                The first thing to note is that the only difference between the complex for simplicial cohomology stated here and the one stated there is the degree $-1$, since they are considering reduced simplicial cohomology.                
                We can now easily see that the complex $\C^\bullet(\triangle, G)$, dual to the homology complex $\C_\bullet(\triangle, G)$, has the same groups as $\vC^\bullet(\{D(x_i)\}, G)$: in every degree $j\geq0$ this group is $G^{\triangle_j}$.                
                As for the maps, it is again easy to see that the map for vector spaces that Miller and Sturmfels describe in their book, can be written instead for $G$ and, when we restrict their complex to the non negative degrees, it is exactly the map of the \v{C}ech complex described above.
            \end{proof}

            \begin{corollary}\label{corollary:sheaf-cohomology}
                Since $\{D(x_i)\}$ is an acyclic covering of $\Spec^\bullet M_\triangle$, for every sheaf of abelian groups, the cohomology in Theorem~\ref{thm:simplicial-cohomology} is also equal to the sheaf cohomology $\H^i(\Spec^\bullet M_\triangle, G)$.
            \end{corollary}
            
            The previous Corollary relates sheaf cohomology, \v{C}ech cohomology and simplicial cohomology in the case of the punctured spectrum. The next one, extends these results to any open subset of the spectrum.
            
            \begin{corollary}\label{corollary:cech-covering-nerve-simplicial-cohomology}
                Let $U\subseteq\Spec^\bullet(M_\triangle)$ be an open subset minimally covered by the covering $\mathscr{V}=\{D(F_1), \dots, D(F_j)\}$ for some $F_1, \dots, F_j\in\triangle$. We have
                \[
                \H^i(U, G)\cong{\vH^i({\mathscr{V}}, G)}\cong{\H^i(\nerve(\mathscr{V}), G)}.
                \]
            \end{corollary}
            \begin{proof}
                The first isomorphism is easy because $D(F)$ is affine and $G$ is a sheaf of abelian groups, hence acyclic on the covering $\mathscr{V}$. Moreover, thanks to Theorem~\ref{thm:simplicial-cohomology}, we know that $\H^i(\nerve(\mathscr{V}),G) = \H^i(\Spec^\bullet M_{\nerve(\mathscr{V})}, G)$. It is enough to show that $\Spec^\bullet M_{\nerve(\mathscr{V})}\cong U$ as topological spaces. This is easily done thanks to the correspondences above between prime ideals and faces of the simplicial complex.
            \end{proof}
                        
        We want to understand now the localization of a simplicial binoid $M_\triangle$ at a face $F$ with the help of the \emph{link complex} of $\triangle$ at $F$. The link complex is the simplicial complex on $V \setminus F$ consisting of all faces $G \subseteq V \setminus F$ with the property that $F \cup G$ is a face of $\triangle$. For a face $F\in\triangle$, we once again denote by $x_F$ the sum $\sum_{v \in F} x_v $ and by $(M_\triangle)_{x_F} = \left(M_\triangle\right)_{\sum_{i\in F} x_v}$ the binoid localized at the variables corresponding to the elements of $F$. Our goal is to prove that $M_{\lk_\triangle(F)}\wedge {(\ZZ^F)}^\infty \cong {(M_\triangle)}_{x_F}$.
        
        \begin{lemma}\label{lemma:map_Zface-localization}
            Let $F$ be a face of $\triangle$. There exists an injective binoid homomorphism
            \[  {(\ZZ^F)}^\infty \stackrel{\psi}{\longrightarrow} {(M_\triangle)}_{x_F} \, ,
             \sum_{v\in F} n_v v \longmapsto  \sum_{v\in F} n_v x_v \, .  \]
        \end{lemma}
        \begin{proof}
            This is clear since $x_v$ is a unit on the right for any $v\in F$ and since $M_\triangle$ is semifree.
        \end{proof}

        \begin{lemma}\label{lemma:map_link-localization}
            Let $F$ be a face of $\triangle$. There exists an injective binoid homomorphism
            \[     M_{\lk_\triangle(F)} \longrightarrow  {(M_\triangle)}_{x_F} \, ,
            \sum_{w\in G} n_w x_w \longmapsto \sum_{w\in G} n_w x_w \, .    \]
        \end{lemma}

\begin{theorem}\label{theorem:localization-simplicial-binoid-multiple}
    For any face $F\in \triangle$ there is an isomorphism
    \begin{equation}\label{equation:localization-simpicial-binoid-wedge}
    (M_\triangle)_{x_F}\cong M_{\triangle'}\wedge(\ZZ^F)^\infty \, ,
    \end{equation}
    where $\triangle'=\lk_\triangle(F)$.
\end{theorem}
\begin{proof}
	By Lemma~\ref{lemma:map_Zface-localization}, Lemma~\ref{lemma:map_link-localization} and the universal property of the smash product (see \cite[Proposition 1.8.10]{Boettger}), we get a commutative diagram
    \[
    \begin{tikzcd}[baseline=(current  bounding  box.center), cramped]
    M_{\lk_\triangle(F)}\wedge {(\ZZ^F)}^\infty\drar["\zeta"] & {(\ZZ^F)}^\infty \lar["i", swap] \dar["\psi"]\\
    M_{\lk_\triangle(F)} \uar["j"] \rar["\varphi", swap] & {(M_\triangle)}_{x_F} \, .
    \end{tikzcd}
    \]
    We can then explicitly describe $\zeta$ in general as
    \[
    \zeta\left(\sum_{w\in G} n_w x_w\wedge \sum_{v\in F} m_vx_v\right)=\sum_{w\in G} n_w x_w + \sum_{v\in F} m_vx_v \, .
    \]
   This map is injective because the maps $\psi$ and $\varphi$ are injective themselves and $G$ and $F$ are disjoint.
    
    Moreover, it is surjective because every element $f\in (M_\triangle)_{x_F}$, $f \neq \infty$, has a unique description (the binoid is semifree) with respect to the semibasis,
    \[    f=\sum_{w\in V \setminus F} n_w x_w + \sum_{v\in F} n_v x_v \, .  \]
    Here, the indices $w$ with $n_w \neq 0$ belong to some $G \in \lk_\triangle(F) $, else $f$ were $\infty$. Hence $f=\zeta(\sum_{w\in G} n_w x_w \wedge \sum_{v\in F} n_v x_v)$ for $G=\{x_w\mid n_w\neq 0\}\in \lk_\triangle(F)$.
\end{proof}

\subsection{The punctured \texorpdfstring{\v{C}}{C}ech-Picard Complex}

Our goal is to compute the cohomology of the sheaf of units $\O^*$ on the punctured spectrum of a binoid $M$. A special feature in the simplicial case is that one can decompose this sheaf into easier sheaves depending only on one vertex.   

\begin{definition}\label{definition:extension-by-zeros}
    Let $\ver$ be a vertex of a simplicial complex $\triangle$, let $M=M_\triangle$ be the corresponding simplicial binoid and let $j_\ver :D(\ver)\longrightarrow \Spec^\bullet M$ be the  open embedding. Let  $\ZZ$ be the constant sheaf on $D(\ver)$. We denote by $\O^*_{\ver }$ the extension of $\ZZ$ by zero along $j_\ver $, that is the sheafification of the presheaf on $\Spec^\bullet M$
    \[
    \begin{tikzcd}[baseline=(current  bounding  box.center), cramped, row sep = 0ex,
    /tikz/column 1/.append style={anchor=base east},
    /tikz/column 2/.append style={anchor=base west},
    ampersand replacement=\&]
    \G:U \rar[mapsto] \& \left\{\begin{aligned}
    &\ZZ, && \text{ if } U\subseteq D( \ver),\\
    &0, && \text{ otherwise.}
    \end{aligned}\right.
    \end{tikzcd}
    \]
\end{definition}

    Since $\ver$ is a unit of $\O_M$ on $D(\ver)$, we think of $\ZZ$ on $D( \ver)$ as multiples of $\ver$.

\begin{remark}\label{remarkvertexunitstalk}
    We can easily describe the stalk of $\O^*_{\ver}$ at $\p$ as
    \[
    \left(\O^*_{\ver}\right)_\p=\varinjlim_{\p\in U}\O^*_{\ver}(U)=\left\{\begin{aligned}
    & \ZZ , && \text{ if } \p\in D(\ver ),\\
    &0, &&\text{ otherwise.}
    \end{aligned}\right.
    \]
\end{remark}

\begin{example}
	Consider the binoid $M=(x, y\mid x+y=\infty)$. Its punctured spectrum is $U=\{\langle x\rangle, \langle y\rangle\}$, that we can cover with $D(x)$ and $D(y)$, that have empty intersection. Indeed, $\O^*_x(D(x))=\ZZ$ and $\O^*_x(D(y))=0$, so $\O^*_x(U)=\ZZ$. This shows that the sheafification is needed and that $\O^*_\ver(U) = \ZZ$ does not mean that $\ver$ is a unit on $U$, but that there is a nonempty component of $U$ where $\ver$ is a unit.
\end{example}

\begin{lemma}\label{theorem:value-sheaf}
	Let $M=M_\triangle $ be the simplicial binoid associated to the simplicial complex $\triangle$ on the vertex set $V $ and let $F \subseteq V $. Then
	\begin{equation}
	\O_M^* \left(\displaystyle\bigcap_{i \in F}D(x_i)\right)\cong \left\{\begin{aligned}
	\ZZ^F && \text{ if } F\in\triangle, \\
	0 && \text{ otherwise.}
	\end{aligned} \right.
	\end{equation}
	and
	\begin{equation}
	\O^*_{\ver}\left(\displaystyle\bigcap_{i\in F} D(x_i)\right)\cong\left\{\begin{aligned}
	\ZZ    && \text{ if } F\in\triangle \text{ and } \ver \in F ,\\
	0 && \text{ otherwise\,.}
	\end{aligned}\right. 
	\end{equation}
\end{lemma}
\begin{proof} Thanks to Lemma~\ref{lemma:intersection-open-subsets} the statement is clear if $F$ is a non-face, so assume that $F$ is a face. We have
	\[    \O_{M}\left(\displaystyle\bigcap_{i\in F} D(x_i)\right) \cong 
	(M_\triangle)_{x_F} \cong   M_{\lk_\triangle(F)}\wedge (\ZZ^F)^\infty     \]
	by  Theorem~\ref{theorem:localization-simplicial-binoid-multiple}. 
	Since any simplicial binoid is positive, we obtain our result.
	
	The second statement follows from the definition and since the sheafification does not affect the affine non-empty subsets.
\end{proof}

\begin{lemma}\label{unitvertexembedding}
	There exists a morphism of sheaves
	\[
	\begin{tikzcd}[baseline=(current  bounding  box.center), cramped, row sep = 0ex,
	/tikz/column 1/.append style={anchor=base east},
	/tikz/column 2/.append style={anchor=base west}]
	\O^*_{\ver} \rar & \O^*_{M}\, .
	\end{tikzcd}
	\]
\end{lemma}
\begin{proof}
	There is a morphism of presheaves $
	\begin{tikzcd}[baseline=(current  bounding  box.center), cramped, row sep = 0ex,
	/tikz/column 1/.append style={anchor=base east},
	/tikz/column 2/.append style={anchor=base west}]
	\G \rar & \O^*_{M}
	\end{tikzcd}
	$
	because if $U \subseteq D(\ver)$ then $\G(U)=\ZZ$ and $\ver$ is a unit of $\O_M(U)$, so we just send $1\mapsto \ver $. If $U \nsubseteq D(\ver)$ then the value of the presheaf is $0$, so we send it to $0$ in $\O_M^*(U)$.
	Thanks to the universal property of the sheafification, we have the following diagram
	\[
	\begin{tikzcd}[baseline=(current  bounding  box.center), cramped]
	\G \rar\dar & \O^*_{M} \\
	\O^*_{\ver}\urar
	\end{tikzcd}
	\]
	that yields the desired morphism $\O^*_{\ver} \longrightarrow \O^*_{M}$.
\end{proof}

\begin{theorem}\label{theorem:semifree-decomposition-of-sheaf}
    For a simplicial complex $\triangle$ on a vertex set $V$ there exists an isomorphism of sheaves
    \[
    \begin{tikzcd}[baseline=(current  bounding  box.center), cramped, row sep = 0ex,
    /tikz/column 1/.append style={anchor=base east},
    /tikz/column 2/.append style={anchor=base west}]
    \displaystyle\bigoplus_{\ver \in V }\O^*_{\ver } \rar & \O^*_{M}
    \end{tikzcd}
    \]
    on the spectrum of $M=M_\triangle$.
\end{theorem}
\begin{proof}
    This map exists because it is induced component-wise by the maps obtained in Lemma~\ref{unitvertexembedding}, applied to the different vertices. In order to show that it is an isomorphism, recall from \cite[Exercise II.1.2]{hartshorne1977algebraic} that a morphism $\varphi:\F\longrightarrow \G$ between two sheaves on a topological space $X$ is an isomorphism if and only if it is an isomorphism at the stalks.    
    Let $\p\in\Spec M$. Thanks to Lemma~\ref{proposition:minimal-open-set-prime-ideal}, there exists a unique minimal open subset that contains $\p$, namely the fundamental open subset $D\left(\sum_{\ver \notin \p}  \ver \right)$.    
    On the right hand side, we have by Lemma~\ref{theorem:value-sheaf}
    \[ \O^*_{M, \p} = \O^*_{M}\left(D\left(\sum_{\ver \notin \p} \ver \right)\right)\cong\ZZ^r\, ,  \]
    where $r$ is the cardinality of $\{ \ver \mid \ver \notin \p\}$. On the other hand
    \[    \left(\displaystyle\bigoplus_{\ver \in V} \O^*_{\ver}\right)_\p\cong \displaystyle\bigoplus_{\ver \in V} \left(\O^*_{\ver}\right)_\p\cong \displaystyle\bigoplus_{\ver \in V } \O^*_{ \ver }\left(D\left(\sum_{ \ver  \notin \p}  \ver \right)\right)\cong\ZZ^r \, ,    \]
    and the morphism between them is the identity.
\end{proof}

Thanks to this Theorem, we know that we can decompose the \v{C}ech complex of $\O^*$ associated to the covering $\{D(x_i)\}$ as the direct sum of the \v{C}ech subcomplexes of this decomposition.

\begin{corollary}\label{theorem:split-cech-picard-groups}
  The \v{C}ech complex for the sheaf of units on the combinatorial affine covering of a simplical binoid as defined in Definition~\ref{definition:cech-picard-complex} is given by 
	\begin{equation}
 \vC^j=\bigoplus_{\begin{subarray}{c}F\in\triangle \\ \left|{F}\right|=j+1\end{subarray}}\ZZ^{F}=\bigoplus_{F\in\triangle_j}\ZZ^{F} \cong \bigoplus_{F\in\triangle_j}\ZZ^{j+1}.
	\end{equation}
	There exists the decomposition
	\begin{equation}
	\vC^j =\bigoplus_{\ver \in V}\vC^j_{\ver},
	\end{equation}
	where
	\begin{equation}
	 \vC^j_{\ver} =  \bigoplus_{F \in \triangle_j , \ver \in F }\ZZ \, . 	\end{equation}   
\end{corollary}
\begin{proof}
   This follows from Lemma~\ref{theorem:value-sheaf} and Theorem~\ref{theorem:semifree-decomposition-of-sheaf}.
\end{proof}

\begin{remark}
	An element $ \alpha= \alpha_{(\ver,F )} \in \vC^j_{\ver} $ is just a collection of integers indexed by $(\ver,F)$, where $\ver \in F$ and $F$ contains $j+1$ elements. Suppose that $V=[n]$ is ordered. Under the map in the \v{C}ech complex, it is sent to $\beta_{(\ver,G)}$, where for $G$ a face containing $\ver$ with $j+2$ elements. We have
	\[ \beta_{(\ver,G)} = \sum_{k = 0}^{j+1} (-1)^k \alpha_{(\ver, G \setminus \{k\} ) } \, . \]
	For $k=\ver$, the entry is zero. For computing the cohomology we can always reorder and assume that $\ver = n$ is the last vertex.
\end{remark}

\begin{lemma}
	\label{vertexunitslink}
	For $j \geq 1$, we have
\[ \vC^j_\ver = \bigoplus_{v \in F,\, \left|{F}\right|=j+1 } \ZZ    = \bigoplus_{H \in \lk_{\triangle}(\ver ),\, \left|{H}\right| = j } \ZZ = C^{j-1} ( \lk_{\triangle}(\ver ) , \ZZ)= \operatorname{Hom} (C_{j-1} ( \lk_{\triangle}(\ver ) ) , \ZZ )  , \, \]
where $C_{j-1} $ denotes the group of $j-1$-chains of the simplicial complex $ \lk_{\triangle}(\ver ) $.
For $j=0$, this statement is also true if we interpret $C_{-1} ( \lk_{\triangle}(\ver ) , \ZZ)  $ as $\ZZ$ (given by the empty set).

If $\ver =n$, then this identification respects also the maps in the \v{C}ech complex and the maps for computing simplicial cohomology.
\end{lemma}
\begin{proof}
All statements are clear from the definitions.
\end{proof}

\subsection{Cohomology}

Summing up what we did until now, we can produce the following Theorem that allows us to compute sheaf cohomology in terms of the reduced simplicial cohomology.

\begin{theorem}\label{theorem:cohomology-simplicial-complex}
    Let $\triangle$ be a simplicial complex on the finite vertex set $V$. We have the following explicit formula for the computation of the cohomology groups of its \v{C}ech-Picard complex
    \begin{equation}\label{formula:cohomology-simplicial-binoid}
    {\H^j\left(\Spec^\bullet M_\triangle, \O^*_{M_\triangle}\right)}\cong\bigoplus_{\ver \in V}\widetilde{\H}^{j-1}\left(\lk_\triangle( \ver ), \ZZ\right)
    \end{equation}
    for $j \geq 0$, where $\widetilde{\H}$ is the reduced simplicial cohomology.
\end{theorem}
\begin{proof}
	We set $V=\{1, \ldots ,n \}$ and denote the corrsponding elements in $M_\triangle$ by $x_i$.
    We can use the open subsets defined by the variables $\{D(x_i)\}$ as a \v{C}ech covering for $\Spec^\bullet M_\triangle$ (Proposition~\ref{proposition:minimal-covering}) to compute the sheaf cohomology of the sheaf of units. By Theorem~\ref{theorem:semifree-decomposition-of-sheaf}, there exists an isomorphism of sheaves
     \[ \O^*_{M_\triangle}=\O^*_{x_1}\oplus \O^*_{x_2}\oplus \dots \oplus \O^*_{x_n} \, .  \]
    The cohomology of $\O^*_{x_i}$ can also be computed with this covering. In Lemma~\ref{vertexunitslink}, we observed that
    \[    \vC^\bullet\left(\Spec^\bullet M, \O^*_{x_i}\right) = \widetilde{\C}^{\bullet-1}\left(\lk_\triangle(i), \ZZ\right) \, ,    \]
    where, for $i=n$, the identifications also respect the mappings. Since for the computation of cohomology we can always reorder $V$, we obtain that
    \[
    \begin{aligned}
    \H^j\left(\Spec^\bullet M_\triangle, \O^*_{M_\triangle}\right)&=\H^j\left(\Spec^\bullet M_\triangle, \bigoplus_{i\in V}\O^*_{x_i}\right)\\
    &=\bigoplus_{i\in V}\H^j\left(\Spec^\bullet M_\triangle, \O^*_{x_i}\right)\\
    &=\bigoplus_{i\in V}\widetilde{\H}^{j-1}\left(\lk_\triangle(i), \ZZ\right)\,  .
    \end{aligned}
    \]  
\end{proof}

\begin{corollary}\label{corollary:cohomology-simplicial-binoid}
    The $0$-th and the first cohomology groups are always free and they have the form
    \begin{align*}
    \H^0(\Spec^\bullet M_\triangle, \O^*)&=\ZZ^{\#\{0-\dim\text{ facets of $\triangle$}\}},\qquad
    \H^1(\Spec^\bullet M_\triangle, \O^*)=\ZZ^r,
    \end{align*}
    where
    \begin{align*}
    r&=\sum_{v\in V}\rk({\widetilde{\H}^0(\lk_\triangle(v), \ZZ)})
    =\sum_{v\in V} \rk({\H^0(\lk_\triangle(v), \ZZ)}) -\#\{\text{$0$-$\dim$ non-facets of $\triangle$}\}\, .
    \end{align*}
\end{corollary}

It follows that the local combinatorial Picard group of a simplicial complex $\triangle$ is $0$ if and only if all links $\lk_\triangle(v)$ are connected. This is true for the simplices, but also for many other examples, see the next Section. We also mention that in cohomological degree $\geq 2$ torsion can occur.

\begin{corollary}
    $\H^j\left(\Spec^\bullet M_\triangle, \O^*_{M_\triangle}\right)=0$, for $j\geq \dim\triangle+1$.
\end{corollary}

\subsection{Examples}

\begin{corollary}\label{proposition:graphs-general}
 Let $\triangle = (V, E)$ be a simple graph. For an isolated vertex $v$ we have
 \[ \H^0 \left(\Spec^\bullet M_\triangle, \O^*_{v}\right) = \ZZ \text{ and } \H^1 \left(\Spec^\bullet M_\triangle, \O^*_{v}\right) = 0 \, , \]
 and for a nonisolated vertex we have
     \[ \H^0 \left(\Spec^\bullet M_\triangle, \O^*_{v}\right) = 0 \text{ and } \H^1 \left(\Spec^\bullet M_\triangle, \O^*_{v}\right) = \ZZ^{ \operatorname{deg} (v) -1}  \, , \] 
     where $ \operatorname{deg} (v) $ denotes the degree of the vertex, i.e. the number of adjacent edges. Moreover,
     \[ \H^0\left(\Spec^\bullet M_\triangle, \O^*_{M_\triangle}\right) =\ZZ^s \, ,\] where $s$ is the number of isolated vertices of the graph and
     \[ \Pic^{\loc} (M_\triangle) = \H^1\left(\Spec^\bullet M_\triangle, \O^*_{M_\triangle}\right)= \ZZ^r \, , \]
     where $r = \displaystyle \sum_{\begin{subarray}{c}
     	v\in V \\
     	v\text{ not isolated}
     	\end{subarray} } \left( \operatorname{deg} (v)  - 1 \right) $. All higher cohomologies vanish.

    \end{corollary}

\begin{example}
	We describe explicitely the line bundles given on the punctured spectrum of a simplicial complex in the situation where $\{u,v\}$ and  $\{v,w\}$ are faces and $\{u,v,w\}$ is a non-face.
	We look at $\O_v^* $, and the \v{C}ech cohomology class $ av \in D(u+v), bv \in D(v+w)$ and $0$ on all other intersections. The cocycle condition is fulfilled, as $D(u+v+w)= \emptyset$. Using $\O_v^*(D(v))=\ZZ$, we can normalize to $b=0$. For $a \neq 0$, this gives a nontrivial line bundle.  
	
	Let $ c,d \in \NN_+$ be such that $a=c-d$ and consider the $M$-set $S$ given by (we set $v=x_1$, $u=x_2$, $w=x_3$, other variables are allowed, but irrelevant)
	\[
	S=\left(e_i, 1 \leq i \leq n \midd \begin{aligned}
	e_1+x_2&=e_2+cx_1,\\
	e_1+x_3&=e_3+dx_1,\\
	e_1+x_j&=e_j + x_1\text{ for } j \geq 3 \\
	e_i+x_j&=e_j+x_i \text{ for } i,j \geq 2
	\end{aligned}\right) \, .
	\]
	This is invertible, since after localizing at any $x_j$, we can eliminate the $e_i$, $i \neq j$, and we see that $e_j$  is a generator of $S_{x_j}$ over $M_{x_j}$. The emptyness of $D(u+v+w)$ ensures that there are no further relations. If we work with the generators $f_1=e_1-dx_1$ and $f_j=e_j-x_j$, we get the transition functions $f_1-f_2 = e_1 -dx_1 - e_2 +x_2 = (c-d) x_1 = ax_1   $ on $D(x_1+x_2)$ and $0$ everywhere else.	
\end{example}

\begin{remark}
  	 From Remark~\ref{degreezerounit}, we get the following short exact cohomology sequence  in the case of a connected graph,
  	\[  0 \longrightarrow \ZZ \longrightarrow  \Pic^{\operatorname{proj} } M   \longrightarrow \Pic^{\loc} M \longrightarrow H^1(U,\ZZ)  \longrightarrow 0 \, . \]
  	Here, $H^1(U,\ZZ) = H^1(\triangle,\ZZ)$ and $ \Pic^{\operatorname{proj} } M \cong \ZZ^{ E } $.  	
  	One should think of $\operatorname{Proj} M$ as a union of combinatorial projective lines whose intersection pattern is a copy of the graph. The second identity is given by sending an edge $e=\{u,v\} $ to the cohomology class given by $u-v$ on $D(u+v)$ and $0$ on the other intersections. These classes are non-trivial in $\Pic^{\operatorname{proj} } M$ though they might be trivial in $\Pic^{\loc} M$. Since $\Pic^{\loc} M \cong \ZZ^{2 |E|- |V| } $, we have an exact sequence
  	\[  0 \longrightarrow \ZZ \longrightarrow  \ZZ^{|E|}   \longrightarrow \ZZ^{2 |E|- |V| } \rightarrow H^1( \triangle ,\ZZ)  \longrightarrow 0 \,  \]
  	from which we can deduce that the cyclomatic number (the rank of $H^1(\triangle,\ZZ)$) equals
  	$ -|E| +1 + (  2 |E|- |V|  ) =|E|- |V| +1  $. \end{remark}
  
\begin{example}
The simplical complex $\triangle$ given by $u,v,w,z$ with facets $\{x,y,z\}$ and $\{y,z,w\} $ corresponds to the binoid $\NN^4/z+w$. All links are connected, hence the local combinatorial Picard group is trivial by Corollary~\ref{corollary:cohomology-simplicial-binoid}, but it is not a simplex. Since $H^1(\triangle, \KK^*) = 0$,  the local Picard group of $\KK[\triangle]$ is also trivial by Theorem~\ref{theorem:cohomology-stanley-reisner} below. 
\end{example}

\begin{example}
We consider the following pictured simplicial complex $\triangle$:

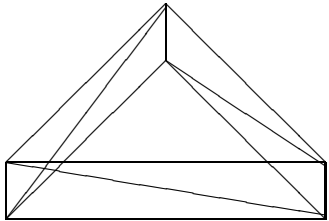
\setlength{\unitlength}{0.75mm}
 
\begin{picture}(60,40)

\put(30,00){\line(1,0){56}}
\put(30,10){\line(1,0){56}}

\put(30,00){\line(1,1){28}}

\put(30,10){\line(1,1){28}}

\put(58,28){\line(1,-1){28}}

\put(58,38){\line(1,-1){28}}

\put(30,00){\line(0,1){10}}

\put(86,00){\line(0,1){10}}

\put(58,28){\line(0,1){10}}

\put(30,10){\line(6,-1){56}}

\put(30,00){\line(3,4){28}}

\put(58,28){\line(3,-2){28}}

\end{picture}

The triangles on the three rectangles belong to the complex, but the triangle on the bottom and on the top not. The link for each vertex consists of four points which are connected by a chain of edges. Hence the local combinatorial Picard group is trivial by Corollary~\ref{corollary:cohomology-simplicial-binoid}. This simplicial complex can be contracted to a circle. Hence $H^1(\triangle, \KK^*) = \KK^*$ and it follows from Theorem~ \ref{theorem:cohomology-stanley-reisner} below that the local Picard group of $\KK[\triangle]$ is not trivial.
\end{example}

    \section{From Combinatorics to Algebra}\label{section:injections}
    
    In this section, we investigate the relations between the local Picard group of binoids and the local Picard group of binoid $\KK$-algebras, where $\KK$ denotes a fixed base field. It will turn out that in many cases the algebraic local Picard group decomposes into a combinatorial part and a part depending on the base field.      
   
   \subsection{Units in $K[M]$}\label{units}
   
    The faithful functor
    \begin{equation}\label{functor:binoids-rings}
    \begin{tikzcd}[baseline=(current  bounding  box.center),
    /tikz/column 1/.append style={anchor=base east},
    /tikz/column 2/.append style={anchor=base west}, 
    row sep = 0pt]
    \KK[\hspace{1em}]: \mathrm{Binoids}\rar & \KK\text{-}\mathrm{Algebras}  \, ,\\
    M\rar[mapsto] & \KK[M] \, ,
    \end{tikzcd}
    \end{equation}
    induces other functors of spectra, sheaves and cohomology groups, that we are going to exploit in what follows. For a fixed binoid $M$, we get a functor from (finitely generated) $M$-sets to (finitely generated) $\KK[M]$-modules:
    \begin{equation}\label{functor:mset-kmmodules}
    \begin{tikzcd}[baseline=(current  bounding  box.center),
    /tikz/column 1/.append style={anchor=base east},
    /tikz/column 2/.append style={anchor=base west}, 
    row sep = 0pt]
    \KK[\hspace{1em}]: M\text{-}\mathrm{Sets}\rar & \KK[M]\text{-}\mathrm{Modules} \, , \\
    S \hspace{1em\rar[mapsto]}& \KK[S] \, ,
    \end{tikzcd}
    \end{equation}
    where $\KK[S]$ is the free $\KK$-module on $S \setminus \{p\}$, and $p$ is the special point of $S$, together with the natural action of $\KK[M]$. This functor is again faithful and it respects localizations (see \cite[Corollary  3.2.8]{Boettger}).
    
    From now on, assume that $M$ is \emph{torsion-free up to nilpotence} and \emph{cancellative}.
    
    \begin{lemma}\label{binoidintegralalgebraintegral}
        If $M$ is a torsion-free, integral and cancellative binoid, then $\KK[M]$ is an integral domain.
    \end{lemma}
    \begin{proof}
        This is a standard result in the toric setting.
    \end{proof}
    
    \begin{lemma}\label{lemma:prime-in-M-iff-prime-in-KM}
      An ideal $\p$ is a prime ideal of $M$ if and only if $\mathfrak{P}=\KK[\p]$ is a prime ideal of $\KK[M]$.
    \end{lemma}
    \begin{proof}
        $\Longleftarrow$ trivial, since $\KK[\p]\cap M=\p$. $\Longrightarrow$ $\mathfrak{P} $ is prime if and only if ${\KK[M]}/{\mathfrak{P}}$ is an integral domain. We know that $\KK\left[{M}/ {I}\right] \cong {\KK[M]}/{\KK[I]}$ for any ideal, and ${M}/{\p}$ is integral because $\p$ is a prime ideal. We can apply Lemma~\ref{binoidintegralalgebraintegral} and get the result.
    \end{proof}

    We come now to the splitting behavior of the sheaf of units.        
    
    \begin{lemma}\label{lemma:units-modulo-prime}
    	If $\p$ is a prime ideal of a  cancellative binoid $M$ that is torsion-free up to nilpotence , then
    	\[        \KK \left[M /  \p \right]^* \cong \KK \left[ (\ZZ^l)^\infty \right ]^* =  \ZZ^l    \times  K^*       \]
    	for some $l$.
    \end{lemma}
    \begin{proof}
    	Modulo the prime ideal we are in a toric setting, where this is known.
    \end{proof}

    \begin{lemma}\label{lemma:torsion-free-cancellative-reduced-algebra}
    	Let $M$ be a reduced, torsion-free, cancellative binoid. Then $\KK[M]$ is reduced.
    \end{lemma}
    
    \begin{theorem}\label{theorem:split-algebra}
    	Let $M$ be a reduced, torsion-free, cancellative binoid and let $\KK[M]$ be its binoid algebra. Then
    	\[        \left(\KK[M]\right)^* = M^* \oplus  \KK^*   \, .        \]
    \end{theorem}
    \begin{proof}
    	What we have to prove is that, under these hypothesis, any unit is a product of a monomial and a unit in the field. On the binoid side, since by definition $M_+=M \setminus M^*$, there is an isomorphism
    	\[        \left(M^*\right)^\infty \cong {M}/{M_+}\, .        \]
    	Let $\p$ be a prime ideal of $M$. Since $\p\subseteq M_+$, there are maps
    	\[
    	\begin{tikzcd}[cramped]
    	M\rar["\pi_\p"]&  {M}/{\p}\rar["\pi_{M_+}"]& {M}/{M_+}=(M^*)^\infty \, ,
    	\end{tikzcd}
    	\]
    	and since $(M^*)^\infty\subseteq M$, we have a map $\sigma$ going the other way
    	\[
    	\begin{tikzcd}[cramped]
    	M \rar["\pi_\p"]& {M}/{\p}\rar["\pi_{M_+}"]&(M^*)^\infty\arrow[ll, out=-10, in=190, overlay, pos=0.07, "\sigma"]
    	\end{tikzcd}
    	\]
    	such that the composition $\pi_{M_+}\circ \pi_\p\circ \sigma$ is the identity on $(M^*)^\infty$. Thanks to the functor from binoids to algebras, we get maps for the rings
    	\[
    	\begin{tikzcd}[cramped]
    	\KK[M]\rar& \KK\left[ {M}/{\p}\right]\rar&\KK\left[(M^*)^\infty\right]\arrow[ll, out=-10, in=190, overlay, pos=0.07]
    	\end{tikzcd}
    	\]
    	that induce maps of groups
    	\[
    	\begin{tikzcd}[cramped]
    	\KK[M]^*\rar& \KK\left[{M}/{\p}\right]^*\rar&\KK\left[(M^*)^\infty\right]^*\arrow[ll, out=-10, in=190, overlay, pos=0.07]
    	\end{tikzcd}
    	\]
    	that again compose to the identity on the right. Let $P$ be a unit in $\KK[M]$. Then
    	\[        P= \lambda_\nu T^\nu +\sum_{\mu\in M_+}\lambda_{\mu}T^\mu        \]
    	with $\nu\in M^*$, since $\KK[M^*]\cong {\KK[M]}/{\KK[M_+]}$ and the statement is true for $\KK[M^*]$, thanks to Lemma~\ref{lemma:units-modulo-prime}. If we apply the first map $\pi_\p$ to $P$, we get
    	\[       \lambda_\nu T^\nu + \sum_{\mu \in M_+, \, \mu \notin \p } \lambda_{\mu} T^\mu \in \KK\left[ {M}/{\p} \right]^* \, .  \]
    	We can apply Lemma~\ref{lemma:units-modulo-prime} to obtain that this has to be a monomial. So, in particular, $\sum_{\mu\in M_+}\lambda_{\mu}T^\mu\in\KK[\p]$ for all minimal prime ideals $\p$. This means that
    	\[
    	\sum_{\mu\in M_+}\lambda_{\mu}T^\mu \in \bigcap_{\begin{subarray}{c}
    		\p \text{ minimal}\\
    		\text{prime of } M
    		\end{subarray}} \KK[\p]=\nil(\KK[M])\, .
    	\]
    	Since $\KK[M]$ is reduced, thanks to Lemma~\ref{lemma:torsion-free-cancellative-reduced-algebra}, its nilradical is trivial, so
    	\[
    	\sum_{\mu\in M_+}\lambda_{\mu}T^\mu=0
    	\]
    	and
    	\[
    	P = \lambda_\nu T^\nu \, .\qedhere
    	\]
    \end{proof}
    
    \begin{remark}\label{remark:units-non-reduced}
    	If $M$ is torsion-free and cancellative but not reduced, then the algebra is not reduced. Still, we can split its units as
    	\[
    	\left(\KK[M]\right)^*= M^*  \oplus       \KK^*\oplus   (1+\n)
    	\]
    	where $\n$ is the nilradical of $\KK[M]$. Indeed, in the above proof, we get that
    	\[
    	N=\sum_{\mu\in M_+}\lambda_{\mu}T^\mu
    	\]
    	is nilpotent. Then $ \lambda_\mu T^\mu + N = \lambda_\mu T^\mu \left( 1 + \frac{N}{ \lambda_\mu T^\mu} \right) \in 1 + \n$.
    \end{remark}

    \begin{example}
    	Consider the non-cancellative and torsion binoid $M=(x, y \mid 2x=x+y, 2y=x+y)$, whose ring is $R= {\KK[X, Y]}/{(X^2-XY, Y^2-XY)}$. The element $X-Y$ is nilpotent in $R$, since $(X-Y)^2=X^2-2XY+Y^2=0$, but does not come from a nilpotent element in $M$, since $M$ is reduced. So $1+X-Y$ is algebraically invertible but it is not the product of a combinatorially invertible element and a field unit, and this shows that the units of a binoid ring are not always the direct sum of combinatorial units and the units of the field.
    \end{example}

    \subsection{Relations between $\Spec M$ and $\Spec \KK[M]$}

    The functor \eqref{functor:mset-kmmodules}, together with Lemma~\ref{lemma:prime-in-M-iff-prime-in-KM}, gives rise to an injection
    \begin{equation}\label{injection:specm-speckm}
    \begin{tikzcd}[baseline=(current  bounding  box.center),
    /tikz/column 1/.append style={anchor=base east},
    /tikz/column 2/.append style={anchor=base west}, 
    row sep = 0pt]
    i: \Spec M\rar[hook] & \Spec \KK[M] \, . \\
    \p \hspace{1em}\rar[mapsto] & \KK[\p] \, .
    \end{tikzcd}
    \end{equation}
    
    \begin{lemma}\label{lemma:i-continuous}
        $i$ is a continuous map between the two spaces equipped with the respective Zariski topologies.
    \end{lemma}
    
    \begin{lemma}\label{ipreimage}
    	Let $P=\sum \alpha_\mu T^\mu \in\KK[M]$ where $T^\mu$ are the monomials corresponding to $\mu \in M$. Then $i^{-1}(D(P))=\bigcup D(\mu )$, where the union runs over all $\mu$ with $\alpha_\mu \neq 0$.      
    \end{lemma}  
    
    \begin{lemma}\label{lemma:U-intersect-SpecM}
        For any non empty open subset $\emptyset \neq U\subseteq\Spec\KK[M]$, the intersection $U\cap i(\Spec M)$ is non empty.
    \end{lemma}
    \begin{proof}
        If $M$ is integral then $\langle0\rangle=i(\langle\infty\rangle)\in U$. If $M$ is non integral, consider a minimal prime ideal $\mathfrak{P}\in\Spec \KK[M]$. Then $\mathfrak{P}=i(\p)$, see \cite[Corollary 3.3.5]{Boettger}. Thanks to the correspondence between prime ideals in $\KK[M]$ that contain $\mathfrak{P}$ and prime ideals in ${\KK[M]}/{\mathfrak{P}}$, and thanks to the fact that ${\KK[M]}/{\mathfrak{P}}$ is integral, we can apply the previous case and obtain our result.
    \end{proof}
        
   Since the map in \eqref{injection:specm-speckm} is continuous, we can pushforward a sheaf from the combinatorial spectrum to the algebraic spectrum
    \begin{equation}\label{functor:sheavesM-sheavesKM}
    \begin{tikzcd}[baseline=(current  bounding  box.center),
    /tikz/column 1/.append style={anchor=base east},
    /tikz/column 2/.append style={anchor=base west}, 
    row sep = 0pt]
    \KK[\hspace{1em}]: \mathfrak{Sheaves}_{\Spec M}\rar & \mathfrak{Sheaves}_{\Spec \KK[M]} \, , \\
    \F\rar[mapsto] &  i_*\F \, .
    \end{tikzcd}
    \end{equation}

    \begin{definition}
        Let $\widehat{I}$ be an ideal in $\KK[M]$. We say that $\widehat{I}$ is \emph{combinatorial} if $\widehat{I}=\KK[I]$ for some $I$ ideal of $M$.
    \end{definition}

    \begin{definition}
        Given a prime ideal $\mathfrak{P}\in\KK[M]$, we denote by $\mathfrak{P}^{\mathrm{mon}}$ the ideal of $\KK[M]$ generated by the monomials in $\mathfrak{P}$, and by $\mathfrak{P}^{\comb}$ the ideal in $M$ such that $\mathfrak{P}^{\mathrm{mon}}=\KK[\mathfrak{P}^{\comb}]$.
    \end{definition}

    \begin{definition}
        We denote by $\Spec^\bullet \KK[M]$ the \emph{punctured spectrum of $\KK[M]$}, i.e.\
        \[ \Spec^\bullet\KK[M]:=\Spec \KK[M] \setminus \{\KK[M_+]\}  \, . \]
    \end{definition} 
    
    We will be interested in computing the cohomology of some sheaves, both on $\Spec \KK[M]$ and on its punctured version. It is a known result that $\H^1(X, \O^*_X)$ is the Picard group of $X$, i.e.\ the group of invertible $\O_X$-sheaves on $X$, for any ringed space $(X, \O^*_X)$, see for example \cite[Exercise III.4.5]{hartshorne1977algebraic}.
    
    \begin{definition}
        Let $\KK[M]$ be a binoid ring. Its \emph{local Picard group} is the Picard group of the scheme
        \[        (\Spec^\bullet \KK[M], \O_{\KK[M]}\restriction_{\Spec^\bullet \KK[M]}) \]
        and it is denoted by $\Pic^{\loc}(\KK[M])$.
    \end{definition}
    
    Note that $\KK[M]$ is not a local ring, but it is a ring with the prominent maximal ideal $\KK[M_+]$. If $M$ is graded, then this is also the (irrelevant) graded maximal ideal.
    
    \begin{definition}
        The covering $D(X_i)$ of $\Spec^\bullet \KK[M]$ is called the \emph{coordinate affine combinatorial covering}.
    \end{definition}    
    
    \subsection{The Combinatorial topology}

    We want to compute the cohomology of the sheaf of units on $\Spec \KK[M]$ (and open subsets) in the Zariski topology, using the results from the combinatorial setting. However, the splitting for the units of monoid rings from Section \ref{units} does not hold for the sheaf of units in the Zariski topology, not even for the affine line. To remedy this situation we introduce a new topology on $\Spec \KK[M]$, that is in between the topology on $\Spec M$ and the Zariski topology on $\Spec \KK[M]$. This topology will be often coarse enough to still have the splitting but also fine enough to compute the cohomology in the Zariski topology.

    \begin{proposition}
        The collection of sets $\{D(\mathfrak{A})\}$, with $\mathfrak{A}$ a combinatorial ideal, is a topology on $\Spec\KK[M]$.
    \end{proposition}
    
        \begin{definition}\label{definition:combinatorial-topology}
        The topology $\{D(\mathfrak{A})\}$ is called the \emph{combinatorial topology} of $\Spec \KK[M]$ and it is denoted by $\Top_{\Spec \KK[M]}^{\comb}$.
    \end{definition}
    
    \begin{corollary}\label{corollary:basis-combinatorial-topology}
        The collection of sets $\{D(P)\}$, with $P$ a monomial in $\KK[M]$, is a basis for the combinatorial topology.
    \end{corollary}

    \begin{remark}
        The combinatorial topology is, in general, a really coarse topology, since $\Spec\KK[M]$ equipped with it is not even $T_0$. In fact, two points $\mathfrak{P}$ and $\mathfrak{Q}$ in $\Spec \KK[M] $ have exactly the same combinatorial neighbourhoods if and only if they contain the same set of monomials, $\mathfrak{P}^{\mathrm{mon}}=\mathfrak{Q}^{\mathrm{mon}}$.
    \end{remark}
        
        \begin{remark}
        We have the commutative diagram of continuous maps
        \[
        \begin{tikzcd}
        \Spec M\rar["i"] \drar["j", swap]& \Spec \KK[M]_{\mathrm{Zar}}\dar["\lambda"]\\
        &\Spec \KK[M]_{\comb}
        \end{tikzcd}
        \]
        where $\lambda$ is the identity (as a set), $i$ is the injection that we proved to be continuous in Lemma~\ref{lemma:i-continuous} and $j$ is the embedding into the space with the combinatorial topology, that is again obviously continuous. Taking the preimages of open subsets along $j$ gives a bijection between the open subsets in the combinatorial topology and the open subsets in the topology of $\Spec M$, which is compatible with intersections and unions.
        
        Given any sheaf $\F$ on $\Spec\KK[M]_{\mathrm{Zar}}$, we get by pushforward along $\lambda$, a sheaf $\lambda_*\F=\F\restriction_{\Top_{\comb}}$ on the combinatorial topology and in particular, $\lambda_*(i_*\O^*_M)=j_*\O^*_M$.
    \end{remark}
    
    \begin{definition}
        Let $\F$ be a sheaf on $\Spec \KK[M]$ equipped with the Zariski topology. We denote by $\F^{\comb}$ the restriction of this sheaf to the combinatorial topology.
    \end{definition}
            
    \begin{lemma}
        For any sheaf of abelian groups $\F$ on $\Spec\KK[M]$ equipped with the combinatorial topology and any  prime ideal $\mathfrak{P}$ of $\KK[M]$, we have that
        \[   \F_{\mathfrak{P}}\cong \F_{\KK[\mathfrak{P}^{\comb}]} \, .    \]
    \end{lemma}
    \begin{proof}
        $\mathfrak{P}$ and $\KK[\mathfrak{P}^{\comb}]$ have the same combinatorial neighbourhoods, i.e.\ $\mathfrak{P}\in U $ if and only if $\KK[\mathfrak{P}^{\comb}]\in U$ for $U \in \Top^{\comb} $.
    \end{proof}
    
    On the combinatorial topologiy, we can extend the splitting results from Section \ref{units} on the sheaf level.    
    
    \begin{proposition}\label{proposition:split-sheaves-comb-top}
        Let $M$ be a reduced, torsion-free and cancellative binoid and let $\KK[M]$ be its binoid algebra. Then
        \[ (\O_{\KK[M]}^*)^{\comb}    \cong (i_*\O^*_M)^{\comb}      \oplus    (\KK^*)^{\comb}  \, , \]
        where $\KK^*$ is the constant sheaf.
    \end{proposition}
    \begin{proof}
        We have a natural injective sheaf homomorphism in the Zariski topology
        \[
        \begin{tikzcd}[cramped]
          i_*\O^*_M   \oplus  \KK^*    \rar & \O_{\KK[M]}^* \, ,
        \end{tikzcd}
        \]
        because every element in $ i_*\O^*_M(U)  \oplus \KK^*(U) $ is trivially a unit in $\O^*_{\KK[M]}(U)$ for any $U$, and in particular in the combinatorial topology. We have to show that the map
        \[
        \begin{tikzcd}[cramped]
        (i_*\O^*_M)^{\comb}   \oplus   (\KK^*)^{\comb}    \rar & (\O_{\KK[M]}^*)^{\comb}
        \end{tikzcd}
        \]
        is an isomorphism. It is enough to show the isomorphism on the combinatorial affine open subsets $D(P)$, with $P$ a monomial. Indeed,
        \begin{align*}
        \Gamma\left(D(P), (\O_{\KK[M]}^*)^{\comb}\right)&=\Gamma\left(D(P), \O_{\KK[M]}^*\right)\\
        &=\Gamma\left(\Spec\KK[M_P], \O_{\KK[M]}^*\right)=\KK[M_P]^* \, .
        \end{align*}
        Since $M_P$ is again reduced, torsion-free and cancellative, we can apply Theorem \ref{theorem:split-algebra} and obtain a decomposition
        \begin{align*}
        \Gamma\left(D(P), (\O_{\KK[M]}^*)^{\comb}\right)&=\KK[M_P]^*=   \KK^*   \oplus  (M_P)^*  \\
        &=\Gamma\left(D(P), (i_*\O^*)^{\comb}  \oplus  (\KK^*)^{\comb} \right) \, . \qedhere
        \end{align*}
    \end{proof}

    \begin{notation} For any combinatorial open subset $U$ of $\Spec\KK[M]$ and Zariski sheaf $\F$ on $U$, we use the notation $\H^i_{\comb}(U, \F)$ to denote the cohomology of the sheaf $\F^{\comb}$ on $U$, i.e.\ the cohomology of $\F$ in the combinatorial topology.
    \end{notation}
    
    \begin{proposition}\label{proposition:hzar-hcomb}
        If $U=D(\mathfrak{A})=\bigcup D(P)$ is a combinatorial open subset of $\Spec \KK[M]$ and $\{D(P)\}$, with $P$ monomial, is an acyclic covering for the Zariski sheaf $\F$ on $U$, then
        \[        \H^j_{\mathrm{Zar}}(U, \F)=\H^j_{\comb}(U, \F) \, ,        \]
        for all $j\geq 0$.
    \end{proposition}

    \subsection{Pushforwards to the Zariski topology}
    
    We want to compute Zariski cohomology of a sheaf $i_*\F$, where $i:\Spec M \rightarrow \Spec \KK[M]$ for a torsion free cancellative binoid $M$ on open subsets of $\Spec \KK[M]$. For \v{C}ech cohomology this is easy, but in order to show that this coincides with Zariski cohomology, we also need acyclicty results.

    \begin{lemma}\label{lemma:cohomology_pushforward}
        Let $\widetilde{U}$ be a combinatorial open subset of $\Spec \KK[M]$, with a covering $\widetilde{\U}=\{\widetilde{U_j}\}_{j\in J}$ made of combinatorial affine open subsets.
        Let $U=i^{-1}(\widetilde{U})$ be the correspondent open subset of $\Spec M$, covered by $\U=i^{-1}(\widetilde{\U})=\{i^{-1}(\widetilde{U_j})\}_{j\in J}$ and let $\F$ be a sheaf of abelian groups on $U$. Then
        \begin{equation}\label{isomorphism:cohomologyM-cohomologyKM}
        \H^j(U, \F)\cong \vH^j(\widetilde{\U}, i_*\F)
        \end{equation}
        for all $j$.
    \end{lemma}
    \begin{proof}
        Let $U_j=i^{-1}(\widetilde{U_j})$. $\{U_j\}$ defines an acyclic covering of $U$ for $\F$, because they are affine open subsets of $\Spec M$, so its \v{C}ech cohomology computes the cohomology on the left. Moreover, $i_*\F(\widetilde{U_j})=\F(i^{-1}(\widetilde{U_j}))=\F(U_j)$, so the \v{C}ech complexes are the same, $\vC(\widetilde{\U}, i_*\F)=\vC(\U, \F)$, and we get our result.
    \end{proof}
    
    \begin{corollary}
        Let $\F$ be a sheaf of abelian groups on $\Spec^\bullet M$ and let $\U=\{D(X_k)\}$ be the combinatorial covering of $\Spec^\bullet \KK[M]$. Then
        \begin{equation}
        \H^j(\Spec^\bullet M, \F)\cong \vH^j(\U, i_*\F)
        \end{equation}
        for all $j$.
    \end{corollary}

    \begin{lemma}\label{lemma:trivial-cohomology-push-forward}
        Let $\F$ be a sheaf of abelian groups on $\Spec M$ and $\U$ any Zariski covering of $\Spec\KK[M]$. Then
        \begin{equation}
        \vH^j(\U, i_*\F)=0
        \end{equation}
        for all $j\geq 1$.
    \end{lemma}
    \begin{proof}
        The preimage of the covering $i^{-1}(\U)$ is a covering of $\Spec M$. In particular, since $i_*\F(U_j)=\F(i^{-1}(U_j))$ for all $U_j\in\U$, the \v{C}ech complexes are the same 
        \[        \C(\U, i_*\F)=\C(i^{-1}(\U), \F) \, .        \]
        Finally, since $\Spec M$ is affine, we know that the cohomology of degree larger than 0, of the combinatorial complex, is zero, and so, it is the one of the pushforward.
    \end{proof} 
    
    \begin{corollary}\label{corollary:H1pushforward_is_0}
        $\H^1(\Spec \KK[M], i_*\F)=0$
    \end{corollary}
    \begin{proof}
        From \cite[Exercise III.4.4]{hartshorne1977algebraic} we know that
        \[        \H^1(\Spec \KK[M], i_*\F)=\varinjlim_\U\vH^1(\U, i_*\F) \, ,        \]
        where the limit is taken over all the possible coverings of $X$.
        Assume that there is a non-zero cohomology class $[c]$ in $\H^1(\Spec \KK[M], i_*\F)$. Then there exists a covering that realizes it, i.e.\ $[c]\in\vH^1(\U, i_*\F)$. But this is impossible, thanks to Lemma~\ref{lemma:trivial-cohomology-push-forward}.
    \end{proof}
    
    \begin{lemma}\label{pushforwardstalk}
        $(i_*\F)_{\mathfrak{P}}\cong\F_{\mathfrak{P}^{\comb}}$ for $ \mathfrak{P} \in \Spec \KK[M]$.
    \end{lemma}
    \begin{proof}
        We begin by investigating the stalk of the pushforward
        \begin{align*}
        (i_*\F)_{\mathfrak{P}}=\varinjlim_{\mathfrak{P}\in U}\F(i^{-1}(U))=\varinjlim_{P\notin \mathfrak{P}}\F(i^{-1}(D(P)))=\varinjlim_{P\notin \mathfrak{P}}\F(\cup(D(P_j))) \, ,
        \end{align*}
        where $P=\sum \alpha_jP_j$, $\alpha_j\neq 0$, and $P_j$ are monomials (see Lemma \ref{ipreimage}). Moreover, $P\notin \mathfrak{P}$ implies that there exists $j$ such that $P_j\notin\mathfrak{P}$, and this is true if and only if $P_j\notin \mathfrak{P}^{\comb}$.
        Consider the direct limit
        \[        \varinjlim_{g\notin\mathfrak{P}^{\comb}}\F(D(g)) \, .        \]
        Since $\{g\notin\mathfrak{P}^{\comb}\}\subseteq \{g \notin\mathfrak{P}\}$, there is a natural map
        \[        \varinjlim_{g\notin\mathfrak{P}^{\comb}}\F(D(g))\longrightarrow \varinjlim_{P\notin \mathfrak{P}}\F(\cup(D(P_j))) \, .         \]
        This map is surjective because, given a section in the stalk $s\in (i_*\F)_{\mathfrak{P}}$, there exists a polynomial $P=\sum\alpha_jP_j$ such that $s\in\F(\cup(D(P_j)))$. In particular, one of these $P_j$'s is not in $\mathfrak{P}$ and so not in $\mathfrak{P}^{\comb}$. Let $P_k$ be this monomial, so $s$ comes via the restriction $\F(\cup(D(P_j)))\longrightarrow \F(D(P_k))$ also from a section in $\F(D(P_k))$. As such, it comes from the left, so the map is surjective.
        \\
        This map is also injective because, given $s$ and $t$ in $\displaystyle\varinjlim_{g\notin\mathfrak{P}^{\comb}}\F(D(g))$, if their images are the same in the limit, then, in particular, they are the same on some open subset $D(P_j)$, such that $P_j\notin\mathfrak{P}^{\comb}$. So they were already the same before. This proves that
        \[        (i_*\F)_{\mathfrak{P}}\cong \varinjlim_{g\notin\mathfrak{P}^{\comb}}\F(D(g))=\F_{\mathfrak{P}^{\comb}} \, .\qedhere
        \]
    \end{proof}
    
    \begin{theorem}\label{theorem:push-forward-exact}
        The pushforward of a sheaf of abelian groups along $i$ is exact.
    \end{theorem}
    \begin{proof}
    	For an exact sequence $  0 \rightarrow \F \rightarrow \G \rightarrow   \H \rightarrow 0$ of sheaves of abelian groups on $\Spec M$, we have to show show that the sequence $  0 \rightarrow i_* \F \rightarrow i_* \G \rightarrow  i_*\H \rightarrow 0$ is exact as well. Since exactness is a local property, it is enough to prove this on the stalks. So this follows from Lemma \ref{pushforwardstalk}.
    \end{proof}

    \begin{proposition}\label{proposition:cohomology-pushforward-affine} For any sheaf of abelian groups $\F$ on $\Spec M$, the Zariski cohomology of the pushforward vanishes
        \[        \H^j(\Spec \KK[M], i_*\F)=0        \]
        for all $j\geq 1$.
    \end{proposition}
    \begin{proof}
        We use induction on $j\geq 1$. For $j=1$ this is true by Corollary~\ref{corollary:H1pushforward_is_0}. Let $j\geq 1$. We can embed $\F$ in a flasque sheaf $\G$ on $\Spec M$ and use the exact sequence
        \[
        \begin{tikzcd}[cramped, row sep = 0pt]
        0\rar & \F\rar & \G\rar & \mathcal{Q}\rar & 0
        \end{tikzcd}
        \]
        on $\Spec M$, where $\mathcal{Q}={\G}/{\F}$. We pushforward this sequence along $i$ and, thanks to Theorem~\ref{theorem:push-forward-exact}, we get an exact sequence on $\Spec \KK[M]$
        \[
        \begin{tikzcd}[cramped, row sep = 0pt]
        0\rar & i_*\F\rar & i_*\G\rar & i_*\mathcal{Q}\rar & 0
        \end{tikzcd}
        \]
        that yields a long exact sequence in cohomology (we omit the topological space for ease of notation)
        \[
        \begin{tikzcd}[cramped, row sep = 2ex,
        /tikz/column 3/.append style={anchor=base west}]
        0\rar & \H^0(i_*\F)\rar & \H^0(i_*\G)\rar & \H^0(i_*\mathcal{Q})\arrow[out=-5, in=175, looseness=1.5, overlay, dll]\\
        & \H^1( i_*\F)\rar & \H^1(i_*\G)\rar & \H^1(i_*\mathcal{Q})\arrow[out=-5, in=175, looseness=1.5, overlay, dll]\\
        & \H^2( i_*\F)\rar  & \H^2( i_*\G)\rar & \dots \, . 
        \end{tikzcd}
        \]
        Thanks to \cite[Exercise II.1.16.(d)]{hartshorne1977algebraic}, we know that $i_*\G$ is again flasque, so
        \[
        \H^j(\Spec \KK[M],i_*\G)=0
        \]
        for all $j\geq 1$, and we get isomorphisms
        \[
        \H^j(i_*\mathcal{Q})\cong \H^{j+1}( i_*\F) \, .
        \]
        By the induction hypothesis, the left hand side is $0$, and so is the right hand side.
    \end{proof}

    \begin{remark}
        The cohomology of any sheaf $i_*\F$ on any open combinatorial subset $U$ of $\Spec\KK[M]$ can be computed by \v{C}ech cohomology, using the affine combinatorial covering of $U$, that is the cover given by the fundamental open subsets $\{D(P)\}$, with $P$ monomials. This is true because $D(P)\cong \Spec\KK[M_P]$ and $(i_*\F)\restriction_{D(P)}=i_*(\F\restriction_{D(P)})$ and from Proposition~\ref{proposition:cohomology-pushforward-affine} this cover is acyclic.
    \end{remark}
    
    \begin{remark}
        The Proposition is true in particular for the cohomology of the sheaf $i_*\O^*_M$, that in general is a subsheaf of $\O^*_{\KK[M]}$, so we can compute the cohomology of $i_*\O^*$ on the punctured spectrum using the acyclic covering given by the coordinates $\{D(X_i)\}$.
    \end{remark}
    
    \begin{definition}
        $i_*\O^*_M$ is called the sheaf of \emph{combinatorial units} of $\KK[M]$.
    \end{definition}
    
    \begin{corollary}\label{combinatorialpushforward}
        $\H^i(\Spec^\bullet\KK[M], i_*\O^*_M)\cong \H^i(\Spec^\bullet M, \O^*_M)$.
    \end{corollary}
    
    \begin{corollary}
        If we have $\O^*_{\KK[M]}\cong i_*\O^*_M\oplus\F$ for some sheaf of abelian groups $\F$ in the combinatorial topology, then $\Pic^{\loc}(M)\neq 0$ implies $\Pic^{\loc}(\KK[M])\neq 0$.
    \end{corollary}
    
    \begin{example}\label{Picardnot00}
        	We give an example of an integral binoid, such that the natural map
        	of the combinatorial local Picard group to the local Picard group of the geometric realization over a field is not injective. We consider the monoid $M$ generated by $e,f,g$, subject to the relations
        	\[ 2e=e ,\, 2f=f , \, 2g = e+f \, . \]
        	This monoid is not cancellative because of the idempotent elements and it is also not torsion free, since
        	$2g =e+f = 2(e+f)$, but $g \neq e+f$. The combinatorial prime ideals are 
        	\[ {(\infty), (e,g), (f,g), (e,f,g)}  \,  ,  \]
        	so the combinatorial Krull dimension is $2$ and  the punctured spectrum is covered by $ D(e)$ and $D(f)$, with the intersection $ D(e) \cap D(f) = D(g)$. We determine the Picard-\v{C}ech-complex for the units. We have
        	\[ M_e = g \NN / (4g=2g) ,\]
        	since from $2e=e$, we get $e=0$, as $e$ becomes a unit and so, with the help of $2g = f $, we can eliminate $f$ and only the generator $g$ remains, with the given equation. Similarily, we have $M_f = g \NN /( 4g=2g )$. These binoids are positive. Moreover, we have $M_g  =g \NN /(2g=0) $, where $g$ is a nontrivial unit. Hence the Picard-\v{C}ech-complex is
        	\[ 0 \longrightarrow  0 \oplus 0 \longrightarrow \ZZ/(2)  \longrightarrow 0 \]
        	and $g$ defines a nontrivial cohomology class, showing the existence of nontrivial combinatorial line bundles over the punctured spectrum. 
        	
        	Let now $\KK$ be a field and compute the $\KK$-spectrum of $M$. The values of $e$ and $f$ are either $0$ or $1$, and the possible values of $g$ are  $-1,0,1$. Hence, there are only finitely many $\KK$-points and so $\KK[M]$ has Krull dimension $0$, there is no cohomology at all and the local Picard group is trivial.
        \end{example}

    \subsection{Cohomology of \texorpdfstring{$\O^*$}{Ostar} in the Zariski topology}
    
    Computations of cohomology in the combinatorial topology are only helpful for the computation of the cohomology of $\O^*$ in the Zariski topology if there exist combinatorial coverings which are acylic in the Zariski topology. We will use the follwing result from \cite[Lemma 5]{demeyer1992cohomological}.
    
    \begin{theorem}\label{affinetoricunit}
         If $Y $ is a normal affine toric variety, then
        \[
        \H^p_{\mathrm{Zar}}(Y, \O^*_Y) = 0,
        \] for all $p\geq 1$.       
    \end{theorem}

    Let $\Spec\KK[M]$ be a normal affine toric variety and let $X=\Spec^\bullet \KK[M]$. Then $D(X_j)$ is a normal affine toric variety embedded in $X$, so $\{D(X_j)\}$ defines by Theorem~\ref{affinetoricunit} an acyclic covering for $\O^*$ on $X$. In particular, we can use it to compute Zariski sheaf cohomology via \v{C}ech cohomology on $X$. Moreover, since the $D(X_i)$'s are combinatorial, we can apply Proposition~\ref{proposition:hzar-hcomb} to obtain isomorphisms
    \[    \H^p_{\mathrm{Zar}}(X, \O^*) = \H^p_{\comb}(X, \O^*) = \vH^p(\{D(X_j)\}, \O^*)  \, .  \]
    We can apply Proposition~\ref{proposition:split-sheaves-comb-top} to split the sheaf $(\O^*)^{\comb}$ and obtain
    \begin{align*}
    \H^p_{\mathrm{Zar}}(X, \O^*) &= \H^p_{\comb}(X, \O^*) = \H^p_{\comb}(X, i_*\O^*_M)  \oplus   \H^p_{\comb}(X, \KK^*)   \\
    & = \vH^p(\{D(X_j)\}, i_*\O^*_M)   \oplus  \vH^p(\{D(X_j)\}, \KK^*)  \\
    & = \H^p(\Spec^\bullet M, \O^*_M) \oplus     \vH^p(\{D(X_j)\}, \KK^*)  \\
    & = \H^p(\Spec^\bullet M, \O^*_M) \, ,
    \end{align*}
    where the last steps rests on the integrality.
    
     If we drop the hypothesis of normality, Theorem \ref{affinetoricunit} might not be true, as the following Example shows.
        
    \begin{example}
        Consider the Neil binoid $M=(x, y\mid 2x=3y)$. Its algebra is ${\KK[X, Y]}/ ( X^2-Y^3 ) $ and it defines the curve called the Neil parabola, i.e.\
        \[        C=\Spec {\KK[X, Y]}/( X^2-Y^3 ) \, ,         \]
        that is a toric variety. We already know from Theorem~\ref{theorem:vanishing-combinatorial-cohomology-affine}, that $\H^i (\Spec M, \O^*_M)=0$ for all $i \geq 1$, because it is affine. On the other hand, this variety is not normal and, indeed, $ \Pic(C)=\KK_+ \,  $, where $\KK_+$ is $\KK$ seen as the additive group, see \cite[Exercise 11.15 and 11.16]{eisenbud2013commutative}. So, in this case, we have that
        \[        \KK_+=\H^1(\Spec\KK[M], \O^*)\neq \H^1(\Spec M, \O^*_M)=0 \, .\qedhere        \]
    \end{example}

\section{Stanley-Reisner rings}\label{section:SR-rings}

In this section, we compute the local Picard group of a Stanley-Reisner ring, through the study of the cohomology of the sheaf of units $\O^*$ in the Zariski topology. While doing so, we will also look at the cohomology of higher degrees of the sheaf of units, and we will give a combinatorial description of this cohomology. In order to describe this cohomology group, we first prove that $\H^i(\KK[\triangle], \O^*)=0$ and $\H^i(\KK[\triangle][x, x^{-1}], \O^*)=0$ for $i\geq 1$. This is proved by induction on the complexity of the simplicial complex, where the starting point is Theorem \ref{affinetoricunit} for affine space. We use this to show that we can compute the cohomology of $\O^*$ on $\Spec^\bullet(\KK[\triangle])$ with the combinatorial \v{C}ech covering, where the answers is known thanks to our previous results. Lastly, we will look at the non-reduced monomial case and we will apply results from Section~\ref{section:injections} in order to get some explicit results also in this case.

\subsection{The Spectrum of a Stanley-Reisner ring}

Let $\triangle$ denote a simplicial complex on the finite set $V$ of vertices.

\begin{lemma}\label{Lemma:correspondencefacessubsetsofspec}
    There is a bijective order-preserving correspondence between the faces of the complex $\triangle$ and the linear coordinate subspaces contained in $\Spec \KK[\triangle]$. This correspondence is dimension-preserving, in the sense that the dimension of the linear coordinate subspace is the dimension of the face associated to it plus one. The irreducible components of $\Spec(\KK[\triangle])$ correspond bijectively to the facets of $\triangle$. The component to a facet $F$ is $\Spec(\KK[\mathcal{P}(F)]) \cong {\mathbb A}^{ \#   F}$.
\end{lemma}

Here $\mathcal{P}(F)$ is the power set for $F$, considered as a full simplicial set.

\begin{lemma}\label{Lemma:decompositionofspectrum}
    Let $F$ be a facet of $\triangle$ and $\triangle':=\triangle \setminus\{F\}$. Then \[\Spec(\KK[\triangle])=\Spec(\KK[\triangle'])\cup \Spec(\KK[\mathcal{P}(F)]) \, .\]
\end{lemma}

\begin{lemma}\label{Lemma:intersectionofspecs}
    Let $F$ be a facet of $\triangle$, $\triangle'=\triangle \setminus\{F\}$ and $\triangle''=\triangle'\cap\mathcal{P}(F)$. Then \[\Spec(\KK[\triangle''])=\Spec(\KK[\triangle'])\cap\Spec(\KK[\mathcal{P}(F)]) \, .\]
\end{lemma}

\begin{lemma}\label{Lemma:dimensionintersection}
    Under the hypothesis of the previous Lemma,
    \[ \dim\left(\Spec\left(\KK[\triangle'']\right)\right)\leq\min\left\{\dim\left(\Spec\left(\KK[\triangle']\right)\right), \dim\left(\Spec\left(\KK[\mathcal{P} \left(F\right)]\right)\right)\right\}
    \]
    and equality holds if and only if $F$ is the unique facet of maximal dimension.
\end{lemma}

\begin{corollary}
    If $\triangle$ is a simplex, then we have $\triangle''=\triangle'$ and, equivalently,
    \[
    \Spec\left(\KK[\triangle'']\right)=\Spec\left(\KK[\triangle']\right) \, .
    \]
\end{corollary}

These previous Lemmata will play a key role in the rest of the section.

\begin{remark}
    By intersecting $\Spec\RR[\triangle]$ with the hyperplane $\{\sum X_i=1 \}$ and considering $X_i\geq 0$, we recover a geometric realization of the abstract simplicial complex we have started with.
\end{remark}

\begin{example}
	If the simplicial complex is given by a triangle (three intervals meeting pairwise), then $\KK[\triangle]=\KK [X,Y,Z]/(XYZ)$. Its spectrum, including the intersection with the plane $X+Y+Z=1$, is shown in the following picture.
    
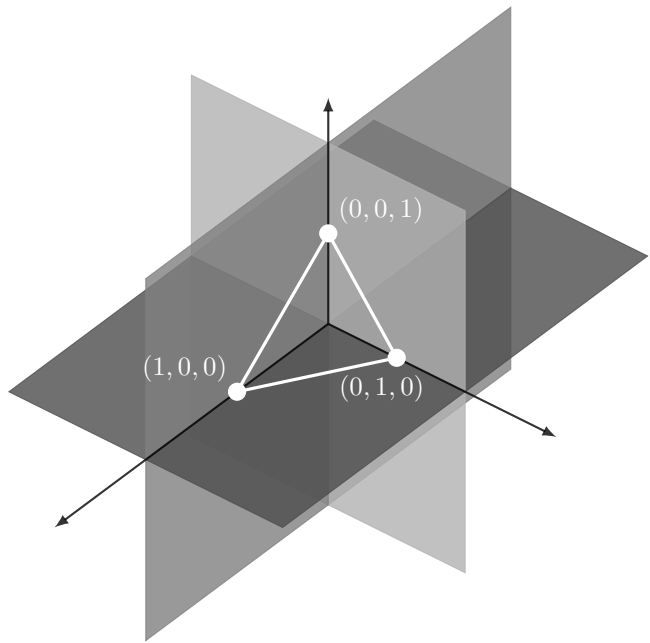
\begin{equation*}
    \begin{tikzpicture}[opacity=0.8, scale=0.6] 
    \filldraw [grey2] (0,-4) -- (0,0) -- (4, 3) -- (4, -1) -- cycle;
    \filldraw [grey3] (0,-4) -- (0,0) -- (-3, 1.5) -- (-3,-2.5) -- cycle;
    \filldraw [grey3] (0,-4) -- (0,0) -- (3, -1.5) -- (3,-5.5) -- cycle;
    \filldraw [grey2] (0,-4) -- (0,0) -- (-4, -3) -- (-4, -7) -- cycle;
    \filldraw [grey1] (0,0) -- (-3,1.5) -- (1,4.5) -- (4, 3) -- cycle;
    \filldraw [grey2] (0,0) -- (0,4) -- (4, 7) -- (4, 3) -- cycle;
    \filldraw [grey1] (0,0) -- (3,-1.5) -- (7,1.5) -- (4, 3) -- cycle;
    \filldraw [grey3] (0,0) -- (0,4) -- (-3, 5.5) -- (-3,1.5) -- cycle;
    \filldraw [grey3] (0,0) -- (0,4) -- (3, 2.5) -- (3,-1.5) -- cycle;
    \filldraw [grey1] (0,0) -- (-3,1.5) -- (-7,-1.5) -- (-4, -3) -- cycle;
    \filldraw [grey2] (0,0) -- (0,4) -- (-4, 1) -- (-4, -3) -- cycle;
    
    \filldraw [grey1] (0,0) -- (3,-1.5) -- (-1,-4.5) -- (-4, -3) -- cycle;
    
    \draw[->, thick, >=latex] (0,0) -- (0,5);
    \draw[->, thick, >=latex] (0,0) -- (-6,-4.5);
    \draw[->, thick, >=latex] (0,0) -- (5,-2.5);
    
    \fill[white, opacity=1] (-2,-1.5) circle[radius=2mm] node[above left] {\small $(1, 0, 0)$};
    \fill[white, opacity=1] (0,2) circle[radius=2mm] node[above right] {\small $(0, 0, 1)$};
    \fill[white, opacity=1] (1.5,-0.75) circle[radius=2mm] node[shift={(-0.2, -0.4)}] {\small $(0, 1, 0)$};
    \draw[white, opacity=1, very thick] (-2,-1.5) -- (0,2) -- (1.5,-0.75) --cycle;
    \end{tikzpicture}
    \end{equation*}
\end{example}

The following statement is a version of Lemma \ref{lemma:intersection-open-subsets} for Stanley-Reisner rings.

\begin{lemma}\label{lemma:covering-algebraic-spectrum}
    Let $F$ be a subset of $V$ and let $D(F)$ denote the affine open subset $D( \prod_{v \in F} X_{v} )= \bigcap_{v \in F}   D(X_{v}) $ of $\Spec \KK[\triangle]$. Then $D(F)\neq \emptyset $ if and only if $F\in\triangle$.
\end{lemma}
\begin{proof}
    This is clear since in the combinatorial topology of $\Spec \KK[\triangle]$, we have the same intersection pattern as in $\Spec M$ and this reflects the combinatorial structure of the simplicial complex.  
\end{proof}

\begin{corollary}\label{corollary:nerve-covering-stanley-reisner}
    The nerve of the covering $\{D(X_i)\}$ of $\Spec^\bullet\KK[\triangle]$ is the simplicial complex itself.
\end{corollary}

\subsection{Acyclicity of the sheaf of units}

In this section, we prove that the covering of the punctured spectrum of a Stanley-Reisner algebra given by the coordinate fundamental open subsets is an acyclic covering for the sheaf of units. In order to show this, we will use the fact that $D(X) \cong \Spec \KK[(M_\triangle)_x] \cong \Spec\KK[\triangle']\times \AA^*$ ($\triangle'$ is the link at the vertex $x$), which we proved in Theorem \ref{theorem:localization-simplicial-binoid-multiple} on the combinatorial level. In particular, we will prove that $\H^j(\Spec\KK[\triangle], \O^*)=0$ for $j\geq 1$.

\begin{lemma}\label{localringunits}
    Let $R$ be a local ring and $\mathfrak{A}$ an ideal of $R$. Then the map $ R^*\rightarrow \left( {R}/ {\mathfrak{A}}\right)^* $ is surjective.
\end{lemma}
\begin{proof}
    This is easily proved because in a local ring, the group of units is the complement of the maximal ideal, and quotients of local rings by ideals are again local rings.
\end{proof}

\begin{proposition}\label{proposition:exact-sequence-units-ideals-trivial-intersection}
	Let $\mathfrak{A}$ and $\mathfrak{B}$ be ideals of a commutative ring $R$ such that $\mathfrak{A}\cap \mathfrak{B}=0$. Let $X=\Spec R$, $Y=\Spec  {R}/{\mathfrak{A}}$ and $Z=\Spec {R}/{\mathfrak{B}}$. Then there exists a short exact sequence of sheaves
	\[
	\begin{tikzcd}[cramped]
	1\rar&\O^*_X\rar["\varphi"]&i_*\O^*_Y\oplus i_*\O^*_Z\rar["\psi"] & i_*\O^*_{Y\cap Z} \rar & 1,
	\end{tikzcd}
	\]
	where $i$ are the inclusion maps and $\varphi(f)=(f\restriction_Y,  f\restriction_Z)$ and $\psi(g, h)=gh^{-1}$.
\end{proposition}
\begin{proof}
	Clearly $X=Y\cup Z$. These maps exist because they are induced by taking the quotients of the involved rings, and the fact that this is a complex is clear. In order to prove the exactness of this sequence, we look at the stalks at a point $\mathfrak{P}$. The surjectivity of $\psi$ follows from Lemma \ref{localringunits}, because the stalks are local rings and $Y\cap Z$ is defined by a quotient of the ring of $Y$ (and of the ring of $Z$). In order to prove injectivity of $\varphi$, we look at it on a stalk
	\[ \O^*_{X,\mathfrak{P}} \stackrel{ \varphi_{\mathfrak{P}} }{\longrightarrow}    (i_*\O^*_Y)_{\mathfrak{P}}\oplus (i_*\O^*_Z)_{\mathfrak{P}}   \, .  \]
	Since $Y=\Spec {R}/{\mathfrak{A}}$ and $Z=\Spec {R}/{\mathfrak{B}}$, we can rewrite this sequence as
	\[ (R_\mathfrak{P})^*  \stackrel{  \varphi_{\mathfrak P} } {\longrightarrow}   \left( {R_\mathfrak{P} } / {\mathfrak{A}}\right)^*\oplus \left(  {R_\mathfrak{P}}/{\mathfrak{B}}\right)^*, \, f \longmapsto (f, f) \, , \]
		where $ {R_\mathfrak{P}}/{\mathfrak{A}}$ and $ {R_\mathfrak{P}}/{\mathfrak{B}}$ denote the quotients via the extended ideals. Consider now $f\in (R_\mathfrak{P})^*$ such that $\varphi_\mathfrak{P}(f)=(1, 1)$. Then $f-1\in\mathfrak{A}$ and $f-1\in\mathfrak{B}$, so $f-1\in\mathfrak{A}\cap\mathfrak{B}=0$, so finally, $f=1$ and this map is injective.
		In order to prove exactness in the middle, we have to show that if $\psi(g, h)=1$, then they both lie in the image of $\varphi$.
		Recall that we have an exact sequence of rings
		\[   0 \longrightarrow  R_\mathfrak{P} \longrightarrow  {R_\mathfrak{P} } / {\mathfrak{A} } \oplus  {R_\mathfrak{P} } / {\mathfrak{B} }   \stackrel{\psi}{  \longrightarrow }  {R_\mathfrak{P}} / {\mathfrak{A}+\mathfrak{B}}  \longrightarrow  0 \, .   \]    
		Let $g, h\in R$ such that $g$ is a unit on $Y$, $h$ is a unit on $Z$ and $\psi(g, h)=1$. This happens if and only if $g=h$ in $Y\cap Z$, because the map $\psi$ sends them to $gh^{-1}$.
		The same holds for $R_\mathfrak{P}$ and the quotients in the sequence above.
		So, there exists $f$ in $R_\mathfrak{P}$ such that $f=g+a=h+b$ in $R_\mathfrak{P}$, with $a\in\mathfrak{A}$ and $b\in\mathfrak{B}$ (where these are the extended ideals in $R_\mathfrak{P}$). What is left to prove is that $f$ is a unit of $R_\mathfrak{P}$.
		Clearly $ f =g $ is invertible modulo $\mathfrak{A}$. Assume that $f$ is not, so it belongs to the maximal ideal $\mathfrak{P}R_\mathfrak{P}$, and if we now go modulo $\mathfrak{A}$, it belongs to $\mathfrak{P} {R_\mathfrak{P}}/{\mathfrak{A} }$, that is again the maximal ideal, and so it would not be invertible. Hence $f$ is invertible and $g, h$ both come from the left, thus the sequence is also exact in the middle.
	\end{proof}

\begin{remark}\label{remark:trivial-intersection-ideals-SR-ring}
    Let $\triangle$ be a simplicial complex and $F$ be one of its facets. Let $X$ be $\Spec \KK[\triangle]$ and $Y$ be the linear coordinate component of $X$ that corresponds to $F$. Let $Z=\overline{X\setminus Y}$ be the union of all the other maximal linear coordinate components in $X$. Then we can apply  Proposition~\ref{proposition:exact-sequence-units-ideals-trivial-intersection} on $Y$ and $Z$.
\end{remark}

\begin{theorem}\label{Thm:vanishingspecstanleyreisner}
    $\H^j(\Spec \KK[\triangle], \O^*)=0$ for all $j\geq 1$.
\end{theorem}
\begin{proof}
    Let $X=\Spec\KK[\triangle]$. We prove the claim by induction on the number of facets of $\triangle$, that correspond to the number of maximal coordinate linear subspaces of $X$.
    If $\triangle$ has only one facet, then it is a simplex and $X\cong\AA^n$ for some $n$, so we get from Theorem~\ref{affinetoricunit}, that $\H^j(X, \O^*)=0$ for all $j\geq 1$.
    \\
    Let now $\triangle$ be any simplicial complex. Consider $Y\cong\AA^m$ a subset of $X$ associated to a facet $F$ of $\triangle$, so $Y$ is a maximal coordinate linear subset of $X$. Let $Z$ be the closure of the complement of $Y$ in $X$, i.e.\ $Z= \overline{X \setminus Y}$. Clearly $Z\cong\KK[\triangle']$ for some simplicial complex $\triangle'$, where $\triangle'=\overline{\left(\triangle \setminus \mathcal{P}(F)\right)}_{\subseteq}$, the subset-closure of the subset of $\triangle$ obtained by removing $F$ and all its subsets from $\triangle$. Clearly, $\triangle'$ has a facet less than $\triangle$, namely $F$.    
    In the same way, $Y\cap Z$ is again a union of coordinate linear subspaces, whose maximal components are the intersection of the maximal components of $Z$ with $Y$, so again coming from another simplicial complex $\triangle''$, which is easier (with smaller dimension and with less facets) than before.        Thanks to Remark \ref{remark:trivial-intersection-ideals-SR-ring}, we know that the radical ideal defining $Y$ and the radical ideal defining $Z$ have trivial intersection in $\KK[\triangle]$. We can then apply Proposition~\ref{proposition:exact-sequence-units-ideals-trivial-intersection} to obtain the short exact sequence of sheaves
    \[
    \begin{tikzcd}[cramped, column sep = 2em]
    1\rar&\O^*_X\rar["\varphi"]&i_*\O^*_Y\oplus i_*\O^*_Z\rar["\psi"] & i_*\O^*_{Y\cap Z}\rar&1.
    \end{tikzcd}
    \]
    When we take cohomology, we obtain the long exact sequence of cohomology on $X$ (we omit the space $X$)
    \[
    \begin{tikzcd}[row sep=2ex, column sep = 2em, cramped]
    \dots\rar&\H^j(\O^*_X)\rar&\H^j(i_*\O^*_Y)\oplus \H^j(i_*\O^*_Z)\rar & \H^j(i_*\O^*_{Y\cap Z})\arrow[dll, overlay, start anchor = east, end anchor = west, to path={ .. controls +(4, -0.8) and +(-2,.8).. (\tikztotarget)}]\\
    &\H^{j+1}(\O^*_X)\rar&\dots
    \end{tikzcd}
    \]
    where, if $j\geq1$, we have that $\H^j(i_*\O^*_Z)=\H^j(Z, \O^*_Z)$ because $Z$ is a closed subset, and this in turn is $0$ by the induction hypothesis, and $\H^j(i_*\O^*_{Y\cap Z})=\H^j(Y\cap Z, \O^*_{Y\cap Z})=0$. Since $Y$ is an affine space, we already know that $\H^j(i_*\O^*_Y)=\H^j(Y, \O^*_Y)=0$. So for $j>1$, we squeeze $\H^j(\O^*_X)$ between two zeros, and this proves that it is zero itself. For $j=1$, we have to look at
    \[
    \begin{tikzcd}[row sep=2ex, column sep = 2em, cramped]
    \H^0(\O^*_X)\rar&\H^0(i_*\O^*_Y)\oplus \H^0(i_*\O^*_Z)\rar & \H^0(i_*\O^*_{Y\cap Z})\rar&\H^1(\O^*_X)\rar & 0 \rar &  \dots \, .
    \end{tikzcd}
    \]
    Since $X$, $Y$ and $Z$ are all defined by Stanley-Reisner ideals, whose units are just the units of the field, this sequence becomes
    \[
    \begin{tikzcd}[row sep=2ex, column sep = 2em, cramped]
     1 \rar & \KK^*\rar & \KK^*\oplus \KK^*\rar["\psi"] & \KK^*\rar & \H^1(\O^*_X) \rar & 0 \rar &  \dots \, .
    \end{tikzcd}
    \]
    But $\psi(s, t)=s^{-1}t$, so it is surjective, so also $\H^1(\O^*_X)=0$.
\end{proof}

\begin{corollary}
   The Picard group of a Stanley-Reisner algebra is trivial.
\end{corollary}

\begin{theorem}\label{Thm:vanishingspecstanleyreisnerA*}
    $\H^j(\KK[\triangle][y_1^{\pm1}, \dots, y_m^{\pm1}], \O^*)=0$ for all $j\geq 1$.
\end{theorem}
\begin{proof}
    The proof is essentially the same as in Theorem~\ref{Thm:vanishingspecstanleyreisner}, except for $\H^1$, because the sequence of global units becomes now
    \[
    \begin{tikzcd}[row sep=2ex, column sep = 2em, cramped]
     (\ZZ)^m   \oplus \KK^*  \rar &( (\ZZ)^m   \oplus \KK^* ) \bigoplus ( (\ZZ)^m   \oplus \KK^*)\rar["\pi"] &  (\ZZ)^m   \oplus \KK^* \rar&\dots  \, ,
    \end{tikzcd}
    \]
    but again it is easy to see that the last map is surjective also on the $\ZZ$'s.
\end{proof}

\subsection{The \texorpdfstring{\v{C}}{C}ech-Picard complex on the punctured spectrum}

Using what we have proved in the previous section, we know that $\{D(X_i)\}$ is an acyclic covering for $\Spec^\bullet \KK[\triangle]$ with respect to the sheaf $\O^*$ and we will describe the groups and the maps appearing in the \v{C}ech complex relative to this covering.

\begin{lemma}\label{theorem:localization-stanley-reisner-combinatorial-multiple} Let $F$ be a face of $\triangle$. Then the localization of the Stanley-Reisner ring of $\triangle$ at $X_F= \prod_{v \in F} X_{v} $ is
    \[    \KK[\triangle]_{X_F}\cong \KK[\triangle'][\ZZ^F] \,  ,     \]
    where $\triangle'=\lk_\triangle(F)$.
\end{lemma}
\begin{proof}
This follows from Theorem \ref{theorem:localization-simplicial-binoid-multiple}.
\end{proof}

\begin{corollary}\label{corollary:covering-cohomology-stanley-reisner}
    The cohomology of $\O^*$ on the punctured spectrum $\Spec^\bullet(\KK[\triangle])$ can be computed using, as \v{C}ech covering, the one given by the fundamental combinatorial open subsets $\{D(X_i)\}$.
\end{corollary}

\begin{remark}
    For a Stanley-Reisner ring, it does not make a difference whether we compute the cohomology of the sheaf of units $\H^j(\Spec^\bullet \KK[\triangle], \O^*)$ on the Zariski or in the combinatorial topology, since the covering $\{D(X_i)\}$ is acyclic in both topologies, and this yields the same \v{C}ech complex. Since the combinatorial topology is simpler, we can restrict to work with it.
\end{remark}

\begin{theorem}\label{theorem:SR-splits-comb-top}
    In the combinatorial topology of $\Spec^\bullet\KK[M_\triangle]$, we have that the sheaf of units splits
    \[    \O^*_{\KK[\triangle]} =       i_*\O^*_{M_\triangle}    \oplus   \KK^* \, , \]
    where $\KK^*$ is the constant sheaf.
\end{theorem}
\begin{proof}
    This is just an application of Proposition~\ref{proposition:split-sheaves-comb-top}, because $\KK[\triangle]=\KK[M_\triangle]$ is reduced and $M_\triangle$ is semifree, so torsion-free and cancellative.
\end{proof}

\begin{remark}
    Since we are able to split $\O^*_M$ into smaller subsheaves (see Theorem~\ref{theorem:semifree-decomposition-of-sheaf}), we can do the same here, and obtain in the combinatorial topology
    \[
    \O^*_{\KK[\triangle]}= i_*\left(\displaystyle\bigoplus_{v \in V}\O^*_{v}\right)   \oplus  \KK^*       =  \left(\displaystyle\bigoplus_{v \in V} i_*\O^*_{v}\right) \oplus   \KK^*        \,.
    \]
\end{remark}

\begin{corollary}\label{corollary:cech-complex-groups-algebraic-case}
    Let $\triangle$ be a simplicial complex on the vertex set $V$, let $\KK[\triangle]$ be its Stanley-Reisner algebra and $\O^*=\O^*_{\KK[\triangle]}$ the sheaf of units. Then
    \begin{equation*}
    \O^*\left(\bigcap_{v\in F}D(X_v)\right)\cong\left\{\begin{aligned}
    & \ZZ^F  \times \KK^*  , &\text{ if } F\in\triangle,\\
    &1,  & \text{ otherwise}.
    \end{aligned}
    \right.
    \end{equation*}
\end{corollary}

\begin{corollary}\label{corollary:cech-complex-split-stanley-reisner}
    The complex for computing \v{C}ech cohomology of $\O^*_{\KK[\triangle]}$ on $\Spec^\bullet\KK[\triangle]$ with respect to the combinatorial covering given by $\{D(X_i)\}$ can be split as a direct sum of the two complexes
    \[
    \vC\left(\{D(X_i)\}, \O^*_{\KK[\triangle]}\right)=\vC\left(\{D(x_i)\}, \O^*_{M_\triangle}\right)  \oplus  \vC\left(\{D(x_i)\}, \KK^*\right)             \,.
    \]
\end{corollary}
\begin{proof}
    The only thing that we have to notice is that 
    \[
    \vC\left(\{D(X_i)\}, \KK^*\right)=\vC\left(\{D(x_i)\}, \KK^*\right) \, ,   \]
    because $\{D(X_i)\}$ and $\{D(x_i)\}$ have the same intersection patterns, thanks to Lemma~\ref{lemma:intersection-open-subsets} and Lemma~\ref{lemma:covering-algebraic-spectrum}.
\end{proof}

\subsection{Cohomology}

We are now ready to sum up our results and give the explicit formulas for computing cohomology of the sheaf of units on the punctured spectrum of a Stanley-Reisner ring.

\begin{lemma}\label{simplicialcomplexconstanttsheaf}
    For any constant sheaf of abelian groups $G$ on $\Spec^\bullet\KK[M_\triangle]$, we have that
    \[
    \vH^j\left(\left\{D(X_i)\right\}, G\right)=\H^j(\triangle, G).
    \]
\end{lemma}

\begin{theorem}\label{theorem:cohomology-stanley-reisner}
    Let $\KK[\triangle]$ be the Stanley-Reisner ring of a simplicial complex $\triangle$ on a finite vertex set $V$. We have the following explicit formula for the cohomology groups of the sheaf of units $\O^*_{\KK[\triangle]}$, restricted to the punctured spectrum $\Spec^\bullet\KK[\triangle]$.
    \begin{equation}\label{formula:cohomology-SR}
    \H^j(\Spec^\bullet(\KK[\triangle]), \O^*_{\KK[\triangle]}) =  \bigoplus_{v\in V} \widetilde{\H}^{j-1}(\lk_\triangle(v), \ZZ) \oplus  \H^j(\triangle, \KK^*)         ,
    \end{equation}
    where $\H^j(\triangle, \KK^*)$ is the $j$-th simplicial cohomology group with coefficients in $\KK^*$.
\end{theorem}
\begin{proof}
	This follows from Corollary \ref{corollary:covering-cohomology-stanley-reisner}, Theorem \ref{theorem:SR-splits-comb-top}, Corollary \ref{combinatorialpushforward}, Theorem \ref{theorem:cohomology-simplicial-complex} and Lemma \ref{simplicialcomplexconstanttsheaf}.
\end{proof}

\begin{example}
	Consider
	\[	\KK[M] =  {\KK[X, Y, Z]} / {(XYZ)} \, . \]
	Its punctured spectrum is covered by $D(X)$, $D(Y)$ and $D(Z)$, and the \v{C}ech complex with respect to this acyclic covering for $\O^*_{\KK[M]}$ is
	\[  \C:  ( \ZZ \oplus \KK^* )\bigoplus (\ZZ \oplus \KK^*)\bigoplus(\ZZ \oplus \KK^*)   \longrightarrow (\ZZ^2 \oplus \KK^*)\bigoplus (\ZZ^2 \oplus \KK^*)\bigoplus(\ZZ^2 \oplus \KK^*) \longrightarrow   1  \, , \]
	and we have the components
	\begin{equation*}
	\begin{tikzcd}[row sep=2ex, 
	/tikz/column 1/.append style={anchor=base east},
	/tikz/column 2/.append style={anchor=base east},
	/tikz/column 3/.append style={anchor=base west}]
	\displaystyle \C(\KK^*): \KK^*\oplus\KK^*\oplus\KK^*\rar["\partial^0_{\KK^*}"] &\KK^*\oplus \KK^*\oplus\KK^* \rar & 1 \, ,  \\
	(\alpha, \beta, \gamma)\rar[mapsto] &\left(\dfrac{\beta}{\alpha}, \dfrac{\gamma}{\alpha}, \dfrac{\gamma}{\beta}\right) \, ,   \\
	\\
	\displaystyle \C(\O^*_M): \ZZ\oplus\ZZ\oplus\ZZ\rar["\partial^0_M"]&\ZZ^2\oplus\ZZ^2\oplus\ZZ^2\, , \rar["\partial^1_M"] & 0 \, , \\
	(a, b, c)\rar[mapsto]& (-a, b, -a, c, -b, c) \, ,
	\end{tikzcd}
	\end{equation*}
	that give us the decomposition. The image of the first map is $(u,v,w)$ with $uv^{-1}w=1$. A nontrivial element is given by $(1,1,-1)$, which gives in the real case the Möbius strip on the triangle. The second part can be split further.
\end{example}

\subsection{The general monomial case}

In this last Section, we show how the general non-reduced monomial case can be handled with our methods. If $I$ is any monomial ideal in $S=\KK[X_1, \dots, X_n]$, then its radical ideal is a squarefree monomial ideal, so $\sqrt{I}=I_\triangle$ for a simplicial complex and the reduction of $R={S}/{I}$  
is the Stanley-Reisner ring $R_{\red}=\KK[M_\triangle]= S/{I_\triangle}$. 
In Theorems~\ref{Thm:vanishingspecstanleyreisner} and \ref{Thm:vanishingspecstanleyreisnerA*}, we proved that the covering of $\Spec^\bullet {S}/{I}$ generated by $\{D(X_i)\}$ is acyclic for $\O^*_{R_{\red}}$ with respect to the Zariski topology.

\begin{theorem}\label{theorem:split-non-reduced-monomial}
	Let $N$ be a binoid whose reduction is a simplicial binoid $M_\triangle$, so that $R=\KK[N]$ is defined by a monomial ideal with reduction $R_{\red} = \KK[\triangle]$. Then we can compute the cohomology of $\O^*_X$ on $X= \Spec^\bullet R$ with the Zariski topology as
    \[   \H^j(X,\O^*_X)  =    \bigoplus_{v\in V}\widetilde{\H}^{j-1}(\lk_\triangle(v), \ZZ)     \oplus      \H^j (\triangle, \KK^* )   \oplus  \H^j(\{D(X_i)\}, 1+\N) \, ,  \]
    where $\N$ is the coherent ideal sheaf of nilpotent elements.
\end{theorem}

\begin{remark}\label{remark:rewrite-non-reduced}
    Since $\O^*_M=\O^*_{M_{\red}}$, we can rewrite the result above as
    \begin{align*}
    \H^j_{\mathrm{Zar}}(X, \O^*_X)&=\H^j(X_{\red}, \O^*_{X_{\red}})\oplus \H^j(\{D(X_i)\}, 1+\N)\\
    &=\H^j(\Spec^\bullet M, \O^*_M)  \oplus \H^j(\triangle, \KK^*) \oplus \H^j(\{D(X_i)\}, 1+\N) \,.
    \end{align*}
\end{remark}

\bibliographystyle{alpha}

\bibliography{bibliography}

\end{document}